\documentclass[preprint]{elsarticle}

\usepackage{amsmath,amssymb,amsthm}
\usepackage{mathrsfs}          
\usepackage{bm}                
\usepackage{xcolor,graphicx}
\usepackage{float}             
\usepackage{enumitem}          
\usepackage{algorithm}         
\usepackage{algpseudocode}     
\usepackage{hyperref}
\usepackage{cleveref}
\usepackage[margin=1in]{geometry}
\usepackage{booktabs}
\usepackage{tikz}
\usetikzlibrary{arrows.meta,positioning,fit,backgrounds,calc}

\algrenewcommand\algorithmicrequire{\textbf{Input:}}
\algrenewcommand\algorithmicensure{\textbf{Output:}}

\newtheorem{theorem}{Theorem}[section]
\newtheorem{lemma}[theorem]{Lemma}
\newtheorem{proposition}[theorem]{Proposition}

\newcommand{\rd}{\mathrm{d}}
\newcommand{\dd}{\,\mathrm{d}}          

\newcommand{\qb}[1]{\bigl(#1\bigr)}          
\newcommand{\qB}[1]{\Bigl(#1\Bigr)}          
\newcommand{\norm}[1]{\left\|#1\right\|}

\newcommand{\R}{\mathbb{R}}
\newcommand{\E}{\mathbb{E}}

\newcommand{\pr}{\mathscr{P}}           

\newcommand{\g}{\gamma}

\renewcommand{\a}{\alpha}
\renewcommand{\b}{\beta}

\renewcommand{\l}{\lambda}
\newcommand{\m}{\mu}

\newcommand{\Sig}{\Sigma}               

\newcommand{\Om}{\Omega}

\newcommand{\sN}{\mathcal{N}}
\newcommand{\sM}{\mathcal{M}}
\newcommand{\sT}{\mathcal{T}}
\newcommand{\sG}{\mathcal{G}}

\newcommand{\inv}{^{-1}}

\journal{arXiv}
\date{}

\begin{document}

\begin{frontmatter}

\title{\textbf{Self-supervised In-context Operator Learning
       for Stochastic Mean-Field Control}}

\author[aff1]{Suyi Gao}
\ead{gao757@purdue.edu}
\author[aff2]{Mo Zhou}
\ead{mozhou366@math.ucla.edu}
\author[aff1]{Rongjie Lai}
\ead{lairj@purdue.edu}
\address[aff1]{Department of Mathematics, Purdue University, West Lafayette, USA}
\address[aff2]{Department of Mathematics, University of California, Los Angeles, USA}

\begin{abstract}
Stochastic mean-field control (MFC) provides a fundamental framework for coordinating large populations of interacting agents under uncertainty, with applications ranging from swarm robotics and systemic-risk management to Schr\"odinger bridges and diffusion-based generative modeling. Existing numerical and deep-learning methods solve one MFC problem instance at a time and must be re-optimized whenever the task changes. In this work, we formulate stochastic MFC as an operator-learning problem and develop, to the best of our knowledge, the first mesh-free, self-supervised neural operator for stochastic MFC. The main challenge is that the diffusion term in the controlled Fokker--Planck equation precludes deterministic transport-map representations. We address this challenge by combining the probability-flow ODE with an invertible normalizing-flow-based transformer, which recasts the dynamics as a deterministic continuity equation and enables closed-form score evaluation through the exact inverse and analytical log-determinant of the normalizing flow, with $\mathcal{O}(d)$ cost per particle for networks of fixed size. Through transformer-based in-context learning, task prompts, represented by compact distribution parameters or raw particle clouds, condition the transport map, enabling a single pretrained operator to solve unseen tasks in one forward pass. The resulting \emph{Normalizing Flow Invertible Solution Transformer} (NFIST) is trained end-to-end by minimizing the stochastic control objective directly, requiring no precomputed numerical solutions for training. We further prove the consistency of the proposed operator-learning formulation with task-by-task optimization. Numerical experiments on stochastic optimal control, Schr\"odinger bridge, systemic-risk control, and obstacle-avoiding path planning demonstrate effective zero-shot generalization while substantially reducing the computational cost of solving large families of stochastic MFC problems.
\end{abstract}


\end{frontmatter}

\section{Introduction}\label{sec:intro}
Stochastic mean-field control (MFC) studies the optimal coordination of a large population of interacting agents whose dynamics are subject to random perturbations. Rather than optimizing the behavior of each individual agent, stochastic MFC seeks a control policy that governs the collective evolution of the population distribution. In the mean-field limit, this evolution is described by a controlled Fokker--Planck equation, leading to a PDE-constrained optimization problem over both the population density and the control field~\cite{lasry2007mfg,huang2006large,carmona2018probabilistic}. Owing to its ability to model uncertainty and large-scale interactions simultaneously, stochastic MFC has emerged as a fundamental framework in optimal control, with applications ranging from swarm and crowd navigation, where agents must reach a destination while avoiding obstacles, to systemic-risk management in financial systems, where a regulator stabilizes the reserves of interacting banks~\cite{carmona2015systemic}, and Schr\"odinger bridge problems, which characterize the most likely stochastic evolution between prescribed endpoint distributions and have recently become a cornerstone of diffusion-based generative modeling~\cite{chen2016relation,chen2021liaisons,song2021score,debortoli2021dsb}.

Despite significant recent progress, existing computational approaches remain largely instance-based. Classical numerical solvers~\cite{Benamou2000ACF,Benamou2015,yu2024fast,Beck2008FISTA} and modern deep-learning-based PDE solvers~\cite{raissi2019pinn,sirignano2018dgm,weinan2018deepritz,han2018solving,ruthotto2020machine,lin2021apacnet,zhou2025deep,zhou2026simulating,zhao2025variational} are designed to solve a single stochastic MFC problem at a time. Whenever the boundary distributions, terminal objectives, obstacle configurations, or cost parameters change, the optimization must be repeated from scratch. In many practical applications, however, one is interested not in a single optimization problem but in an entire family of related control tasks. A robotic swarm repeatedly encounters new environments, a planner must evaluate different design parameters, and a diffusion model requires Schr\"odinger bridges between many pairs of endpoint distributions. Re-solving every instance independently is therefore computationally expensive and prevents real-time deployment. This naturally motivates an operator-learning perspective: instead of approximating the solution of one optimization problem, we seek to learn a solution operator that maps an arbitrary task specification directly to its optimal control and population evolution. Once trained, such an operator can generalize to previously unseen tasks through a single forward evaluation~\cite{lu2021deeponet,li2021fno}.

Recently, a mesh-free self-supervised operator-learning framework for deterministic mean-field games was proposed in~\cite{huang2025unsupervised}, where the population evolution is represented in a Lagrangian formulation through a transformer-parameterized flow map. Their work demonstrates that learning a solution operator, rather than solving each problem instance independently, can dramatically improve computational efficiency while generalizing across families of deterministic mean-field games. Extending this paradigm to \emph{stochastic} mean-field control, however, presents a fundamental challenge. Existing mesh-free operator-learning methods naturally rely on deterministic transport maps that evolve particles according to ordinary differential equations. The diffusion term in the controlled Fokker--Planck equation destroys this deterministic transport structure, making such formulations inapplicable. Although the equivalent probability-flow formulation restores deterministic dynamics, it introduces the score function $\nabla\log p$ of the evolving density into the running cost. Since the density itself is unknown, evaluating this score generally requires density estimation or auxiliary score-learning procedures, both of which incur substantial computational overhead and additional approximation error. Developing a mesh-free operator-learning framework that overcomes these difficulties while remaining fully self-supervised therefore remains an open challenge.

In this work, we address this challenge by combining the probability-flow formulation of stochastic dynamics, invertible normalizing flows, and transformer-based in-context operator learning into a unified framework. The probability-flow ODE~\cite{anderson1982reverse,song2021score,boffi2023probability} reformulates stochastic MFC as an equivalent deterministic optimal control problem governed by a continuity equation~\cite{zhou2025score}. We parameterize the associated transport map using an invertible normalizing flow, whose exact inverse and analytical log-determinant provide closed-form evaluation of the score function through the change-of-variables formula, thereby eliminating density estimation, auxiliary score networks, and numerical inversion. To generalize across tasks, we represent task descriptions—including boundary distributions, obstacle configurations, and control parameters—as prompts that condition the transport map through transformer cross-attention. The resulting architecture, termed the \emph{Normalizing Flow Invertible Solution Transformer} (NFIST), is trained end-to-end by minimizing the expected stochastic control objective over a distribution of tasks in a fully self-supervised manner. At inference time, solving a new stochastic MFC problem requires only encoding its task description and performing a single forward pass through the learned operator.

To the best of our knowledge, this is the first work that formulates stochastic mean-field control as an operator-learning problem and develops a self-supervised neural operator capable of generalizing across families of stochastic MFC tasks. By combining probability-flow dynamics, exact normalizing-flow representations, and prompt-based in-context operator learning, our approach enables efficient zero-shot prediction across broad families of stochastic control problems without task-specific retraining.

\paragraph{Contributions} The main contributions of this paper are summarized as follows.

\begin{enumerate}
    \item \emph{Operator-learning formulation.} We formulate stochastic mean-field control as an operator-learning problem in which each task is specified by probability measures together with Euclidean parameters, and the objective is to learn a solution operator mapping task descriptions to optimal probability flows. The operator is trained entirely in a self-supervised manner by minimizing the expected stochastic control objective over a task distribution. We further prove that this learning objective is statistically consistent with solving individual stochastic MFC problems: under a sufficiently expressive operator class, minimizing the expected objective recovers the optimal solution for almost every task.

    \item \emph{Methodology.} We introduce the \emph{Normalizing Flow Invertible Solution Transformer} (NFIST), which integrates probability-flow dynamics, invertible normalizing flows, and transformer-based prompt conditioning into a unified neural operator. The proposed architecture evaluates every component of the stochastic control objective analytically, requiring neither density estimation, auxiliary score networks, nor numerical inversion. Exact inverse mappings and analytical log-determinants yield efficient   closed-form score evaluation with $\mathcal{O}(d)$ cost per particle for networks of fixed size. NFIST enables zero-shot adaptation to both parametric and nonparametric task descriptions.
    
    \item \emph{Applications and empirical validation.} We demonstrate that a single pretrained NFIST operator generalizes across diverse stochastic MFC problems, including stochastic optimal control, Schr\"odinger bridge problems, systemic-risk control, and obstacle-avoiding path planning. The learned operator achieves high accuracy on unseen tasks without supervision from ground-truth trajectories and further generalizes across continuously varying problem parameters through prompt conditioning, enabling efficient solution of entire families of stochastic control problems using a single pretrained model.
\end{enumerate}

\subsection{Related Work}\label{sec:related}

\paragraph{Operator learning}
DeepONet~\cite{lu2021deeponet} and the Fourier neural operator~\cite{li2021fno} popularized modern neural operator learning, where a neural network approximates mappings between infinite-dimensional function spaces from supervised input--output pairs. In-context operator learning~\cite{yang2023icon} extends this paradigm by conditioning the operator on task-specific prompts, enabling a single model to generalize across families of related operators. More recently, this idea has been generalized to operators on probability measures~\cite{cole2026incontextoperatorlearningspace}. Operator learning has also begun to emerge for mean-field control and games, including deterministic mean-field games~\cite{huang2025unsupervised}, finite-state mean-field games~\cite{hofgard2026operator}, and linear--quadratic mean-field games~\cite{firoozi2025simultaneously}. Our work extends this line to stochastic mean-field control, where the second-order Fokker--Planck equation introduces the score of the evolving density as the main computational challenge. Our work is different from ~\cite{garg2022icl}, which prompt the network  with solved input--output demonstrations. Instead, we prompt it directly with the task description, represented by particle clouds of the task measures together with cost parameters. This sample-based conditioning is closely related to set-encoder and conditional-neural-process architectures~\cite{garnelo2018cnp,zaheer2017deepsets,lee2019settransformer}.

\paragraph{Deep solvers for mean-field control and stochastic control}
Mean-field games and control were introduced in~\cite{lasry2007mfg,huang2006large}; see~\cite{carmona2018probabilistic} for the probabilistic theory and~\cite{RuimengHu2024Recent} for a recent survey of machine learning methods. Classical grid-based approaches include variational and proximal algorithms based on the Benamou--Brenier formulation~\cite{Benamou2000ACF,Benamou2015,yu2024fast,Beck2008FISTA}. In higher dimensions, neural network based approaches for solving single instance of the problem include physics-informed methods~\cite{raissi2019pinn,sirignano2018dgm,weinan2018deepritz}, deep BSDE solvers~\cite{han2018solving}, Lagrangian frameworks~\cite{ruthotto2020machine}, GAN-style training~\cite{lin2021apacnet}, normalizing-flow parameterizations~\cite{huang2023bridging}, particle-based flow matching~\cite{yu2026high} and actor--critic methods for stochastic optimal control~\cite{zhou2021actor,zhou2024actor}.  
Among these, the score-based neural ODE framework of~\cite{zhou2025score} is most closely related to our work, as it also adopts the probability-flow reformulation of stochastic MFC. However, the score is obtained by integrating an auxiliary score ODE along particle trajectories. In contrast, our normalizing-flow parameterization computes the score analytically through the change-of-variables formula, eliminating auxiliary score dynamics altogether. Moreover, whereas existing methods solve one stochastic MFC problem at a time, our objective is to learn a single solution operator that generalizes across an entire distribution of tasks.

\paragraph{Score-based generative models, Schr\"odinger bridges, and flow matching}
Diffusion models generate samples by learning the score of a data distribution and evolving either the reverse-time SDE or the equivalent probability flow ODE~\cite{anderson1982reverse,song2021score}. The Schr\"odinger bridge connects stochastic control with entropic optimal transport~\cite{chen2016relation,chen2021liaisons}; neural approaches include iterative proportional fitting~\cite{debortoli2021dsb}, bridge matching~\cite{shi2023dsbm}, and the Gaussian closed-form solution~\cite{bunne2023schrodinger}. Closely related, flow matching~\cite{lipman2023flow,chen2024flow} and consistency-based models~\cite{kim2023ctm,boffi2024flowmap} learn probability flows by regressing velocity fields or transport maps from data. These methods share our probability-flow viewpoint but differ fundamentally in their learning objective. They estimate scores or velocity fields from samples of a target distribution, whereas our setting has no data distribution to regress toward. Instead, the transport map is optimized directly through the stochastic control objective, and the score is recovered analytically from the normalizing flow via the change-of-variables formula, eliminating score matching or auxiliary score dynamics altogether.

\paragraph{Normalizing flows and coupling architectures}
Normalizing flows provide invertible parameterizations of probability distributions through the change-of-variables formula~\cite{rezende2015flow,papamakarios2021nf,Kobyzev2021}. Among them, coupling-based architectures~\cite{dinh2014nice,dinh2017realnvp} are particularly attractive because they admit exact inverses and analytic log-determinants, while monotone rational-quadratic spline flows~\cite{durkan2019nsf} extend this framework to more expressive one-dimensional transformations. The Jet architecture~\cite{kolesnikov2024jet} combines coupling flows with transformer conditioning for image generation. Our architecture adopts a similar combination but for a different purpose: cross-attention conditions the flow on the MFC task, while self-attention captures interactions among particles, producing a task-conditioned family of transport maps rather than a density model for a single distribution.

\paragraph{Notations}\label{sec:notation}

We use the dot notation for time derivatives: $\partial_t a(x,t) := \frac{\partial a}{\partial t}(x,t)$. All gradient-type operators act spatially unless stated otherwise: $\nabla a := \nabla_x a$,  $\boldsymbol{\nabla} b := \boldsymbol{\nabla}_x\vec b$ (Jacobian matrix), $\nabla\cdot a := \nabla_x\cdot a$ (divergence), $\Delta a := \Delta_x a$ (Laplacian). We write $\sN(\m,\Sig)$ for a Gaussian with mean $\m$ and covariance $\Sig$, $\pr(\R^d)$ for the space of Borel probability measures on $\R^d$, and $\E_{x\sim P}[\cdot]$ for expectation under the measure $P$. For any probability distribution $P$ with density function $p$, we abuse the notation to write the density function of $G_\# P$ as $G_\# p$, where $G$ is a mapping between the ground space of $P$ and $G_\# P$. 

The rest of the paper is organized as follows. Section~\ref{sec:background} reviews stochastic mean-field control, the probability flow formulation, and normalizing flows. Section~\ref{sec:sol} formulates stochastic MFC as an operator-learning problem and presents the proposed NFIST architecture. Section~\ref{sec:results} evaluates the method on stochastic mean-field control benchmarks. Finally, Section~\ref{sec:conclusion} concludes the paper and discusses future directions.

\section{Background}\label{sec:background}
\subsection{Problem Setup: Stochastic Mean-Field Control Problems}\label{sec:mfc}

We consider the controlled state process
\begin{equation}
\dd X_t = v(X_t,t)\dd t + \sqrt{2\g} \dd W_t  \label{equ:sde}
\end{equation}
for a representative agent, where $\g>0$ and $W_t$ is a standard $d$-dimensional Wiener process. In a macroscopic perspective, the population is characterized by a density $p(\cdot,t)$ satisfying the Fokker--Planck (FP) equation. The corresponding stochastic MFC problem seeks an optimal velocity field $v$ that minimizes a global cost functional over a finite horizon $T$:
\begin{align}
    \inf_{p,v}  J(p,v)
    &:= \int_0^T\!\int_\Om L\qb{x,v(x,t)}p(x,t) \dd x\dd t
    + \int_0^T I\qb{p(\cdot,t)}\dd t
    + M\qb{p(\cdot,T)},  \label{equ:mfc_noise} \\
    \text{s.t.}\quad&
    \partial_t p(x,t) + \nabla\!\cdot\!(p(x,t)v(x,t)) = \g\Delta p(x,t),\quad
    p(\cdot,0) = p_0. \label{equ:fp}
\end{align}

Rather than solving~\eqref{equ:mfc_noise} separately for each instance, we learn a solution operator that maps a task specification directly to its optimal population dynamics. Trained over a family of MFC problems, the operator generalizes to previously unseen tasks through a single forward pass.

We remark that the MFC problem~\eqref{equ:mfc_noise} models a centralized planner that optimizes a global population objective. In contrast, a Mean-Field Game (MFG) describes non-cooperative agents, seeking a Nash equilibrium. The two formulations generally differ, and we focus on MFC in this work.

\subsection{Examples}\label{sec:examples}
\paragraph{Example: Stochastic Optimal Control Problem}
In the stochastic optimal control problem, we set $L(x,v)=\|v\|^2$, $I(p)=0$, and $M(p(\cdot,T)) = \int_{\Omega} V(x)\,p(x,T)\,\dd x$. The Eulerian perspective in equations \eqref{equ:mfc_noise} and \eqref{equ:fp} also have the following equivalent form using the Lagrangian perspective: 
\begin{equation}
\min_v\;
\E\left[\tfrac12\int_0^1\norm{v(X_t,t)}^2\dd t+\tfrac12 V(X_1)\right],
\qquad
\rd X_t=v(X_t,t)\,\rd t+\sqrt{2\g}\,\rd W_t,
\quad X_0\sim P_0,
\end{equation}
Here, $V:\mathbb R^d\to\mathbb R$ is the potential function, which may vary from task to task. We therefore consider a family of stochastic optimal control problems parameterized by the potential function $V$. Each task is specified by $\mathcal T = V$,  sampled from a meta-distribution $\mathcal M$. We may also generalize the tasks to have random initial distributions, then the task will become $\mathcal T = (p_0,V)$.  Learning a solution operator over this task family enables zero-shot prediction for previously unseen terminal objectives, without solving the HJB equation explicitly.

\paragraph{Example: Schr\"odinger Bridge}
The Schr\"odinger bridge problem seeks the most likely stochastic evolution between two prescribed endpoint distributions relative to a reference diffusion process, or equivalently, minimizes the Kullback--Leibler (KL) divergence between the controlled and reference path measures. Within the MFC framework~\eqref{equ:mfc_noise}, the Schr\"odinger bridge problem corresponds to minimizing the kinetic energy $L(x,v)=\|v\|^2$ with $I(p)=0$ and $p(\cdot,T)=p_T$; the worked example of Section~\ref{sec:task_disc} instead uses the convention $L(x,u)=\frac12\|u\|^2$, a harmless rescaling of the cost that we state only to fix constants. In practice, enforcing this hard terminal constraint is often inconvenient during training. We therefore replace it by the soft penalty $M(p(\cdot,T)) = \lambda_{\mathrm{KL}} \mathrm{KL}\bigl(p(\cdot,T)\,\|\,p_T\bigr)$, where $\lambda_{\mathrm{KL}}>0$ controls the trade-off between transport cost and terminal accuracy; this direction of the divergence is used because the parameterization of Section~\ref{sec:sol} provides the density of $p(\cdot,T)$ in closed form, so the penalty is computable without density estimation. A task is then specified by the endpoint pair $\mathcal T=(p_0,p_T)$, which follows a meta-distribution $\mathcal M$.

\paragraph{Example: Systemic Risk Model}
The systemic risk model describes the inter-bank borrowing--lending system~\cite{carmona2015systemic}. The state of log-monetary reserves are governed by the one-dimensional SDE
\begin{equation}\label{eq:Sysrisk_state}
\dd X_t = \qb{a(\bar m_t-X_t)+\alpha_t}\dd t+\sigma\dd W_t,
\qquad
\bar m_t=\E[X_t],
\end{equation}
in $1$ dimension. The social cost is
\begin{equation}\label{eq:Sysrisk_cost}
J=\E\int_0^T\qB{\tfrac12\alpha_t^2-q\,\alpha_t(\bar m_t-X_t)
+\tfrac{\varepsilon}{2}(\bar m_t-X_t)^2}\dd t
+\tfrac{c}{2}\E\qb{(X_T-\bar m_T)^2}.
\end{equation}
Writing $v(x,t)=a\qb{\bar m(p)-x}+\alpha(x,t)$ with $\bar m(p)=\int_\Omega x\,p(x)\dd x$ casts the problem into the stochastic MFC formulation~\eqref{equ:mfc_noise}. A task is specified by $\mathcal T=(p_0,\theta)$ with $\theta=(a,q,\varepsilon,c,\sigma)$ sampled from a meta-distribution $\mathcal M$.

\paragraph{Example: Obstacle Avoiding Path Planning}
The MFC framework also applies to swarm path-planning problems, where a population of agents is transported from an initial distribution to a target distribution while avoiding prescribed obstacles. Within the stochastic MFC formulation~\eqref{equ:mfc_noise}, $L(x,v) = \|v(x,t)\|^2$, $M(p(\cdot,T)) = \l_{\mathrm{KL}}\mathrm{KL}\qb{p(\cdot,T) \| p_T}$, and the obstacle cost is
\[
I(p(\cdot,t)) = \int_\Omega Q(x)p(x,t)\,\dd x,
\]
where $Q$ penalizes occupancy of obstacle regions. A task is specified by the tuple $\mathcal T=(p_0,p_T,Q)$ sampled from a meta-distribution $\mathcal M$, consisting of the initial distribution, target distribution, and obstacle layout.

\subsection{Probability Flow Formulation}\label{sec:change_of_variable}

The stochastic MFC problem~\eqref{equ:mfc_noise} is naturally formulated in Eulerian coordinates through the density $p$ and velocity field $v$. For computation, however, a Lagrangian representation is often preferable, since the density evolution can then be computed by transporting particles through a flow map $G(\cdot,t)$ satisfying
\[
p(\cdot,t)=G(\cdot,t)_\#p_0,
\]
avoiding the direct discretization of the Fokker--Planck equation. Such a flow map is unavailable for the stochastic dynamics~\eqref{equ:sde}, whose Brownian motion prevents a deterministic particle flow. To recover an equivalent Lagrangian formulation, we employ the probability flow ODE, which replaces the stochastic dynamics by the deterministic velocity
\begin{align}
    f(x,t):=v(x,t)-\g\nabla\log p(x,t),
\end{align}
whose continuity equation has the same time-marginal densities as the original Fokker--Planck equation.

\begin{lemma}[Probability flow ODE~\cite{anderson1982reverse}]%
\label{lem:prob_flow}
Let $p(\cdot,t)$ be the density of the process~\eqref{equ:sde} with score function $\mathbf{s}(x,t):=\nabla\log p(x,t)$. Define
\begin{align}
    f(x,t) := v(x,t) - \g \nabla\log p(x,t). \label{equ:prob_flow_f}
\end{align}
Then the ODE $\partial_t x_t = f(x_t,t)$ with $x_0\sim p_0$ has the same marginal densities $p(\cdot,t)$ as~\eqref{equ:sde} for all $t\in[0,T]$.
\end{lemma}
We defer the proofs of this and all following results to \ref{sec:appendix_proofs}. Consequently, the stochastic MFC problem can be rewritten as a deterministic control problem for the composed velocity field $f$, with the appearance of the score $\nabla\log p$ in the running cost:
\begin{align}
    \inf_{p,f}  J(p,f)
    &:= \int_0^T\!\int_\Om
        L\qb{x,\, f(x,t)+\g\nabla\log p(x,t)}
    p(x,t)\dd x\dd t
    + \int_0^T I\qb{p(\cdot,t)}\dd t
    + M\qb{p(\cdot,T)},
    \label{equ:mfc_det}
    \\
    \text{s.t.}\quad&
    \partial_t p(x,t) + \nabla\!\cdot\!\qb{p(x,t)f(x,t)} = 0,\quad p(\cdot,0)=p_0.
    \nonumber
\end{align}

The first-order continuity equation admits a Lagrangian representation through a flow map. Instead of parameterizing the velocity field $f$ directly, we parameterize the transport map $G$, which simultaneously determines the density evolution and the agent trajectories. Let
\begin{align}\label{equ:pt_G}
    \partial_t G(x,t)=f\bigl(G(x,t),t\bigr),
    \qquad
    G(x,0)=x.
\end{align}
Then the density is given by the pushforward
\begin{align}
    p(\cdot,t)=G(\cdot,t)_\# p_0.
    \label{equ:flow_map}
\end{align}
Substituting this representation into~\eqref{equ:mfc_det} yields an unconstrained optimization problem over the flow map:
\begin{align}
    \inf_{G}\, J(G;\mathcal{T})
    &:= \int_0^T\!\int_\Om L\qb{G(x,t),\, \partial_t G(x,t) + \g\, \mathbf{s}\qb{G(x,t),t}} p_0(x)\dd x\dd t \notag\\
    &\qquad + \int_0^T I\qb{G(\cdot,t)_\# p_0}\dd t
    + M\qb{G(\cdot,T)_\# p_0},
    \label{equ:lagrangian_cost}
\end{align}
where $\mathbf{s}=\nabla\log p = \nabla \log G(\cdot,t)_\# p_0$ denotes the score function. Here $\mathcal{T}$ is the representation of the tasks, which determines the function $L$, $I$, and $M$. Examples of $\mathcal{T}$ are provided in section~\ref{sec:examples}. This formulation is mesh-free and naturally compatible with particle-based Monte Carlo methods. Moreover, once the flow map $G$ is known, both the density and the score can be recovered through the change-of-variables formula.

\begin{lemma}[Score from pushforward map~\cite{Kobyzev2021}]
\label{lem:score_pushforward}
Let $G(\cdot,t):\R^d\to\R^d$ be a diffeomorphism satisfying $G(\cdot,t)_\# p_0=p(\cdot,t)$. Then, the score at $x_t=G(x,t)$ is
\begin{align}
    \mathbf{s}(x_t,t) = \boldsymbol{\nabla}_x G(x,t)^{-\top}
    \left(\mathbf{s}_0(x)-\nabla_x\log\left|\det\boldsymbol{\nabla}_xG(x,t)\right|\right),
    \label{equ:score_compute_jacobian}
\end{align}
where $\mathbf{s}_0(x)=\nabla\log p_0(x)$.
\end{lemma}

Lemma~\ref{lem:score_pushforward} expresses the score in terms of the Jacobian of the flow map. Directly evaluating~\eqref{equ:score_compute_jacobian} requires the Jacobian $\boldsymbol{\nabla}_x G$, its log-determinant, and solving a linear system involving its transpose. For a generic flow map, explicitly forming and manipulating the Jacobian incurs at least $\mathcal{O}(d^2)$ memory and $\mathcal{O}(d^3)$ computational cost per particle, making score evaluation expensive in high dimensions. The NFIST architecture introduced in Section~\ref{sec:arch} exploits an invertible normalizing-flow construction to avoid these costs. In particular, the coupling structure renders the Jacobian of each layer block triangular with a diagonal transformed block, so the log-determinant is available analytically without matrix factorization. Moreover, the score is computed by a single reverse-mode sweep through the inverse network, without explicitly forming the Jacobian, yielding $\mathcal{O}(d)$ computational cost per particle for a network of fixed hidden width and depth.

\subsection{Recovering the SDE Trajectories from the Probability Flow}\label{sec:sde_recovery}
The probability flow ODE preserves the time marginals of~\eqref{equ:sde} but replaces the stochastic agent paths by deterministic ones: an individual ODE trajectory $t\mapsto G(x,t)$ is not a sample path of the controlled diffusion. The stochastic trajectories are recoverable from the flow map alone. Inverting~\eqref{equ:prob_flow_f} for the control velocity gives
\begin{align}
    v(x,t) = f(x,t) + \g\, \mathbf{s}(x,t),
    \qquad
    f(x,t) = \partial_t G\qb{G\inv(x,t),t},
    \label{equ:v_from_G}
\end{align}
with the score $s$ available in closed form through Lemma~\ref{lem:score_pushforward}. Substituting~\eqref{equ:v_from_G} into~\eqref{equ:sde} yields
\begin{align}
    \dd X_t = \qb{f(X_t,t) + \g\, \mathbf{s}(X_t,t)}\dd t + \sqrt{2\g}\,\dd W_t,
    \qquad X_0\sim P_0,
    \label{equ:sde_recovered}
\end{align}
which is exactly the controlled diffusion~\eqref{equ:sde}: its sample paths are distributed as the original stochastic agent trajectories, not merely marginal-matched. In practice we integrate~\eqref{equ:sde_recovered} by the Euler--Maruyama scheme
\begin{align*}
    X_{t+\Delta t} = X_t + \qb{f(X_t,t)+\g\, \mathbf{s}(X_t,t)}\Delta t + \sqrt{2\g\,\Delta t}\;\xi_t,
    \qquad \xi_t\sim\sN(0,I),
\end{align*}
where each drift evaluation uses only quantities the flow map already provides: one inverse pass for $G\inv$ and the score, and one time derivative of $G$ for $f$. The deterministic parameterization therefore loses no information, the flow map transports the density, and~\eqref{equ:sde_recovered} restores the stochastic trajectories whenever they are required.

\section{Solution Operator Learning}\label{sec:sol}

Classical numerical solvers and most deep-learning PDE methods, including physics-informed neural networks~\cite{raissi2019pinn,sirignano2018dgm} and the Deep Ritz method~\cite{weinan2018deepritz}, solve one problem instance at a time. Consequently, every new choice of the initial distribution, terminal objective, or constraints requires solving the optimization problem again. Operator learning instead approximates the map from problem data to solution function~\cite{lu2021deeponet,li2021fno}, reducing the computational cost over a family of problems so that unseen instances are solved by a single forward pass.

For operator learning, each MFC problem is regarded as a \emph{task} consisting of finitely many probability measures and Euclidean parameters,
\begin{align}
    \sT=(P_1,\ldots,P_m,\theta)
    \in
    \mathfrak T:=\pr(\R^d)^m\times\R^k,
    \label{equ:task_space}
\end{align}
where $P_1,\ldots,P_m$ denote the measure-valued inputs and $\theta\in\R^k$ collects the scalar parameters. The precise choice of these inputs depends on the application: for example, $\sT=(P_0,b)$ for stochastic optimal control, $\sT=(P_0,P_T)$ for the Schr\"odinger bridge, $\sT=(P_0,\theta)$ for the systemic risk model, and $\sT=(P_0,P_T,Q)$ for obstacle avoidance. The corresponding \emph{solution operator} maps each task to its optimal flow map,
\begin{align}
    \sG:\mathfrak T
    \longrightarrow
    C^{2,1}(\R^d\times[0,T];\R^d),
    \qquad
    \sG[\sT]=G.
    \label{equ:sol_op_def}
\end{align}
$G$ satisfies the pushforward relation $G(\cdot,t)_{\#}P_0 = P(\cdot,t)$, giving the constraint of continuity-equation in~\eqref{equ:mfc_det}.

Next, Section~\ref{sec:op_mfc} formulates the operator-learning objective and establishes its consistency with task-by-task optimization. Section~\ref{sec:task_disc} introduces the tokenization of task descriptions into a finite set of prompts and illustrates the construction with a representative example. Finally, Section~\ref{sec:arch} presents the proposed neural architecture that realizes the proposed solution operator.

\subsection{Operator learning for MFC}\label{sec:op_mfc}

Let $\sM\in\pr(\mathfrak T)$ be a meta distribution of MFC tasks. We learn a solution operator $\sG$ by minimizing the expected task objective
\begin{align}
\min_{\sG}\;
\E_{\sT\sim\sM}
J\bigl(\sG[\sT];\sT\bigr),
\label{equ:sol_op_obj}
\end{align}
where $J(\cdot;\sT)$ is the objective~\eqref{equ:lagrangian_cost} associated with task $\sT$. Since the predicted flow map \(G=\sG[\sT]\) determines both the density \(p(\cdot,t)=G(\cdot,t)_\#p_0\) and the velocity \(\partial_t G\), the stochastic MFC objective can be evaluated directly from the network output without solving any auxiliary PDE. The neural operator therefore serves as a surrogate model for the underlying variational optimization problem rather than for its numerical solutions. Training is fully self-supervised: the stochastic control objective itself provides the loss, so no optimal trajectories, value functions, or precomputed numerical solutions are required as training data. Consequently, our approach avoids the substantial offline computational cost of generating large collections of numerical solutions required by supervised operator-learning methods, requiring only samples of tasks drawn from \(\sM\).

The operator is realized through \emph{in-context} learning~\cite{xie2022an}. Each task is encoded as a prompt, namely a finite sequence of context tokens representing the task description \(\sT\) (Section~\ref{sec:task_disc}), which conditions a single neural network with shared parameters through transformer attention. During inference, the network parameters remain fixed, while different task prompts produce different flow maps, enabling the same pretrained model to adapt to and solve previously unseen tasks without retraining.

\begin{proposition}[Consistency with per-task training]
\label{prop:consistency}

Assume that, for $\sM$-almost every task $\sT$, the corresponding single-task problem $\min_G J(G;\sT)$ admits an optimal solution $G_\sT^*$. Suppose further that 
$$\E_{\sT\sim\sM} J(G_\sT^*;\sT) < \infty.$$
Then the minimum possible operator-learning objective equals the expected optimum of the individual tasks:
\begin{align}
\min_{\sG}
\E_{\sT\sim\sM}
J(\sG[\sT];\sT)
=
\E_{\sT\sim\sM}
J(G_\sT^*;\sT).
\label{equ:consistency}
\end{align}
Moreover, an operator $\sG$ is optimal if and only if $\sG[\sT]\in \arg\min_G J(G;\sT)$ for $\sM$-almost every task $\sT$.
\end{proposition}

\begin{proof}
For every task $\sT$, $J(\sG[\sT];\sT) \ge J(G_\sT^*;\sT)$. Taking expectations gives
\[
\E_{\sT\sim\sM}J(\sG[\sT];\sT)
\ge
\E_{\sT\sim\sM}J(G_\sT^*;\sT).
\]
Since this inequality holds for every operator $\sG$, taking the infimum over $\sG$ proves
\eqref{equ:consistency}. Equality holds when
\[
J(\sG[\sT];\sT)=J(G_\sT^*;\sT)
\]
for $\sM$-almost every task, i.e., when $\sG[\sT]$ is an optimal solution of the corresponding single-task problem.
\end{proof}

Proposition~\ref{prop:consistency} shows that operator learning is fully consistent with classical task-by-task optimization. The operator objective is simply the expectation of the individual task objectives, so minimizing it is equivalent to solving every task optimally almost surely under the task distribution. In other words, operator learning amortizes computation across tasks without changing the underlying optimization problem; the only additional challenge is whether the chosen operator class is expressive enough to approximate the optimal task-to-solution map.

\subsection{Discretization of the tasks}\label{sec:task_disc}

Each task~\eqref{equ:task_space} is represented by a finite set of tokens. Euclidean parameters $\theta\in\R^k$ are already finite-dimensional and are used directly as tokens. Examples include the well location $b$ in the stochastic optimal control problem and the cost-weight vector $\theta=(a,q,\varepsilon,c,\sigma)$ in the systemic-risk example. 

Each measure-valued input $P\in\pr(\R^d)$ is represented in one of two ways.

\paragraph{Parametric representation} 
When $P$ belongs to a known family, its parameters are used directly as tokens. For example, a Gaussian distribution $\mathcal N(\mu,\Sigma)$ is encoded by
\begin{align}
    \bigl[\mu\,;\,\operatorname{vech}(\operatorname{chol}\Sigma)\bigr]
    \in
    \R^{d+d(d+1)/2},
    \label{equ:gauss_token}
\end{align}
consisting of the mean and the lower triangular part of the Cholesky factorization of the covariance, and a Gaussian mixture is represented by one token per component.

\paragraph{Non-parametric representation} 
When only samples are available, an i.i.d.\ particle cloud
\[
X=(x^1,\ldots,x^{N_P})\sim P
\]
is used as the prompt, with each particle forming one token.

The two representations trade flexibility against efficiency. Particle clouds require no distributional assumptions and naturally accommodate arbitrary measures, but increase the attention cost with the number of particles and introduce sampling error. Parametric tokens are compact and deterministic, but are limited to the chosen family. In both cases, the resulting token sequence serves as the prompt that conditions the solution operator.

\paragraph{Example: single-well stochastic optimal control}
Consider the task family $\sT=(P_0,b)$, where $P_0=\mathcal N(\mu_0,\Sigma_0)$ and $b\in\R^d$ is the center of the quadratic terminal potential. The objective is
\begin{align}
J(G;\sT)
&=
\frac12
\int_0^T
\E_{x\sim P_0}
\Bigl\|
\partial_tG(x,t)+\gamma \mathbf{s}(G(x,t),t)
\Bigr\|^2dt
+
\frac12
\E_{x\sim P_0}
\|G(x,T)-b\|^2,
\label{equ:soc_functional}
\end{align}
where $s(\cdot,t)$ is the score of the pushforward density $G(\cdot,t)_\#P_0$, computed analytically from $G$ by Lemma~\ref{lem:score_pushforward}.

During training, the expectations are approximated by Monte Carlo and the time integral by a uniform discretization. Given a batch of tasks $\{\sT_k=(P_0^k,b_k)\}_{k=1}^B$, let $x_j^{k,i}=G_k(x^{k,i},t_j)$, $G_k=\sG[\sT_k]$, where $x^{k,i}\sim P_0^k$ and $t_j=jT/n_t$. Replacing $\partial_tG$ by finite differences gives the empirical objective
\begin{align}
\hat J
&=
\frac1B
\sum_{k=1}^{B}
\left[
\frac{\Delta t}{2N}
\sum_{j=0}^{n_t-1}
\sum_{i=1}^{N}
\left\|
\frac{x^{k,i}_{j+1}-x^{k,i}_{j}}{\Delta t}
+\gamma\hat{\mathbf{s}}_{j+1/2}^{k,i}
\right\|^2
+
\frac1{2N}
\sum_{i=1}^{N}
\|x_{n_t}^{k,i}-b_k\|^2
\right],
\label{equ:soc_empirical}\\
\hat{\mathbf{s}}^{\,k,i}_{j+1/2}
    \;&:=\;
    \frac12\,\qB{\mathbf{s}_k\qb{x^{k,i}_{j},\,t_j}+\mathbf{s}_k\qb{x^{k,i}_{j+1},\,t_{j+1}}},
    \notag
\end{align}
where $\mathbf{s}_k\qb{x^{k,i}_{j},\,t_j}$ denotes the numerical evaluation of the score function, obtained using an inverse form of Lemma~\ref{lem:score_pushforward}. Specifically, taking $\nabla_x\log$ of the change-of-variables formula 
\[
    p^k(x,t)=p^k_0\qb{G_k\inv(x,t)}\bigl|\det\boldsymbol{\nabla}_xG_k\inv(x,t)\bigr|
    \ ,
\]
we obtain
\begin{align}
\mathbf{s}_k\qb{x^{k,i}_j,t_j} \;=\; \nabla_x^{\mathrm{autodiff}} \left[ \log p^k_0\Bigl(G_k\inv\qb{x,t_j}\Bigr) + \log\Bigl| \det\boldsymbol{\nabla}_x G_k\inv\qb{x,t_j} \Bigr| \right]_{\,x=x^{k,i}_j}, \label{equ:score_discrete}
\end{align}

The task is encoded using either of the two prompt representations introduced above. In the parametric setting, the prompt consists of the Gaussian token $\bigl[\mu_0;\operatorname{vech}(\operatorname{chol}\Sigma_0)\bigr]$ representing the initial distribution $P_0$, together with the Euclidean parameter $b$ specifying the well center. In the non-parametric setting, these finite-dimensional descriptors are replaced by particle-cloud representations: an i.i.d.\ sample $X_0=(x_0^1,\ldots,x_0^{N_0})\sim P_0$ represents the initial distribution, while a second particle cloud sampled from an isotropic Gaussian centered at $b$ represents the terminal potential. Apart from this change in task representation, the network architecture and the empirical objective~\eqref{equ:soc_empirical} are unchanged. For the non-parametric operator, we additionally regularize the variance, across transported particles, of the terminal log-determinant $\log\bigl|\det\boldsymbol{\nabla}_xG(x,T)\bigr|$. This penalty discourages excessive particle-to-particle variation in the local volume change induced by the learned transport map and stabilizes the resulting empirical density.

The task is encoded in either of the two prompt representations introduced above. The parametric prompt consists of the Gaussian token $\bigl[\mu_0;\operatorname{vech}(\operatorname{chol}\Sigma_0)\bigr]$ together with the Euclidean parameter $b$, while the non-parametric prompt includes i.i.d.\ particle clouds, $X_0=(x_0^1,\ldots,x_0^{N_0})\sim P_0$ for the initial law and a cloud drawn from an isotropic Gaussian centered at the well $b$. 
The network architecture and the objective~\eqref{equ:soc_empirical} are identical in both cases; except the non-parametric operator additionally penalizes the variance across particles of the terminal log-determinant,
\begin{align}
\hat J_{\rm ld}
\;=\;
\frac{\lambda_{\rm ld}}{B}\sum_{k=1}^{B}
\operatorname{Var}_{i\le N}
\Bigl[\,
\log\bigl|\det\boldsymbol{\nabla}_xG_k\qb{x^{k,i},T}\bigr|
\,\Bigr],
\label{equ:ld_penalty}
\end{align}
added to~\eqref{equ:soc_empirical}, where $\operatorname{Var}_{i\le N}$ denotes the empirical variance over the $N$ particles of a task and $\lambda_{\rm ld}>0$. A uniform volume change has a constant log-determinant across particles and has no penalty; the term discourages only spatially excessive per-particle expansion, the mechanism by which the flow could inflate the empirical density particle by particle to lower the Monte-Carlo estimate of the density-dependent terms.

Algorithm~\ref{alg:soc} summarizes the resulting self-supervised training procedure. At each iteration, we sample a fresh batch of tasks together with an independent particle cloud for each task.

\begin{algorithm}[!htbp]
\caption{Self-supervised training of the solution operator on the single-well stochastic optimal control family $\sT=(p_0,V)$.}
\label{alg:soc}
\begin{algorithmic}[1]
    \Require task distribution $\sM$; diffusion $\gamma$; horizon $T$; uniform grid $t_j=jT/n_t$, $j=0,\ldots,n_t$ ($n_t+1$ nodes), with $\Delta t=T/n_t$; tasks per batch $B$; particles per task $N$; learning rate $\eta$; iteration budget
    $n_{\rm iter}$
    \Ensure parameters $w$ of the solution operator $\sG_w$
    \State initialize $w$
    \For{$\mathrm{iter}=1,\ldots,n_{\rm iter}$}
      \State draw $\sT_k=(P_0^k,b_k)\sim\sM$ and $x^{k,i}\sim P_0^k$ i.i.d., $k\le B$, $i\le N$
             \Comment{fresh tasks and particles}
      \State $G_k\leftarrow\sG_w[\sT_k]$
             \Comment{the prompt is encoded once per task}
      \State $x^{k,i}_j\leftarrow G_k\qb{x^{k,i},t_j}$, \; $j=0,\ldots,n_t$
             \Comment{$x^{k,i}_0=x^{k,i}$}
      \State $\mathbf{s}_k\qb{x,t_j}\leftarrow\nabla_x\Bigl[\log p_0^k\qb{G_k\inv\qb{x,t_j}}
             +\log\bigl|\det\boldsymbol{\nabla}_xG_k\inv\qb{x,t_j}\bigr|\Bigr]$
             evaluated at $x=x^{k,i}_j$
             \Comment{Lemma~\ref{lem:score_pushforward}}
      \State $\hat{\mathbf{s}}^{\,k,i}_{j+1/2}\leftarrow\tfrac12
             \qB{\mathbf{s}_k\qb{x^{k,i}_j,t_j}+\mathbf{s}_k\qb{x^{k,i}_{j+1},t_{j+1}}}$, \; $j<n_t$
      \State $\hat J\leftarrow$ the empirical objective~\eqref{equ:soc_empirical}
      \State $g\leftarrow\nabla_w\hat J$;\quad $g\leftarrow g\cdot\min\{1,\,1/\norm{g}\}$
             \Comment{gradient-norm clipping}
      \State $w\leftarrow\mathrm{AdamW}(w,g,\eta)$
      \State periodically re-estimate $\hat J$ on fresh task batches at fixed $w$, and reduce
             $\eta$ when it stalls
    \EndFor
    \State \Return $w$
\end{algorithmic}
\end{algorithm}

The other three problem classes are trained using the same procedure, except that the cost assembled from them changes. For the Schr"odinger bridge, the quadratic terminal cost is replaced by the Kullback--Leibler terminal penalty introduced in Section~\ref{sec:background}; the corresponding experiments are presented in Section~\ref{sec:exp_sb}. For the systemic-risk model, the control is recovered from the probability-flow velocity as $\alpha=\partial_tG+\gamma s-a(\bar m-x)$, see Section~\ref{sec:exp_sr1}. The obstacle-avoiding path-planning problem combines the Kullback--Leibler terminal penalty with the running obstacle-occupancy cost and an entropic collision-avoidance term; the corresponding results are reported in Section~\ref{sec:exp_pp_task}.

\subsection{Architecture: Normalizing Flow Invertible Solution Transformer}\label{sec:arch}

The solution operator must satisfy three requirements: (i) evaluate the flow map $G(\cdot,t)$ at arbitrary time with $G(\cdot,0)=\mathrm{id}$; (ii) admit an efficient inverse $G^{-1}(\cdot,t)$; and (iii) provide the Jacobian log-determinant analytically for score evaluation (see \eqref{equ:score_discrete}). 
To meet these requirements simultaneously, we introduce the \emph{Normalizing Flow Invertible Solution Transformer (NFIST)}, a task-conditioned invertible architecture designed specifically for operator learning in stochastic MFC. NFIST combines a time-conditioned normalizing flow built from affine coupling layers~\cite{dinh2017realnvp,rezende2015flow,Kobyzev2021,papamakarios2021nf} with transformer-based in-context conditioning~\cite{vaswani2017attention,yang2023icon}. The coupling structure provides exact inversion and analytic log-determinants required for score evaluation, while transformer cross-attention conditions the entire flow on the task prompt, turning a single invertible network into a solution operator over a family of stochastic control problems. The time-dependent parameterization further enforces $G(\cdot,0)=\mathrm{id}$ by construction. Figure~\ref{fig:arch} illustrates the resulting architecture.
We note that a related combination of coupling-based normalizing flows and transformer architectures has been explored in~\cite{kolesnikov2024jet} for an image generative model; NFIST adapts these architectural ingredients to a fundamentally different setting, using task-conditioned invertible flows to represent solution operators for stochastic MFC.

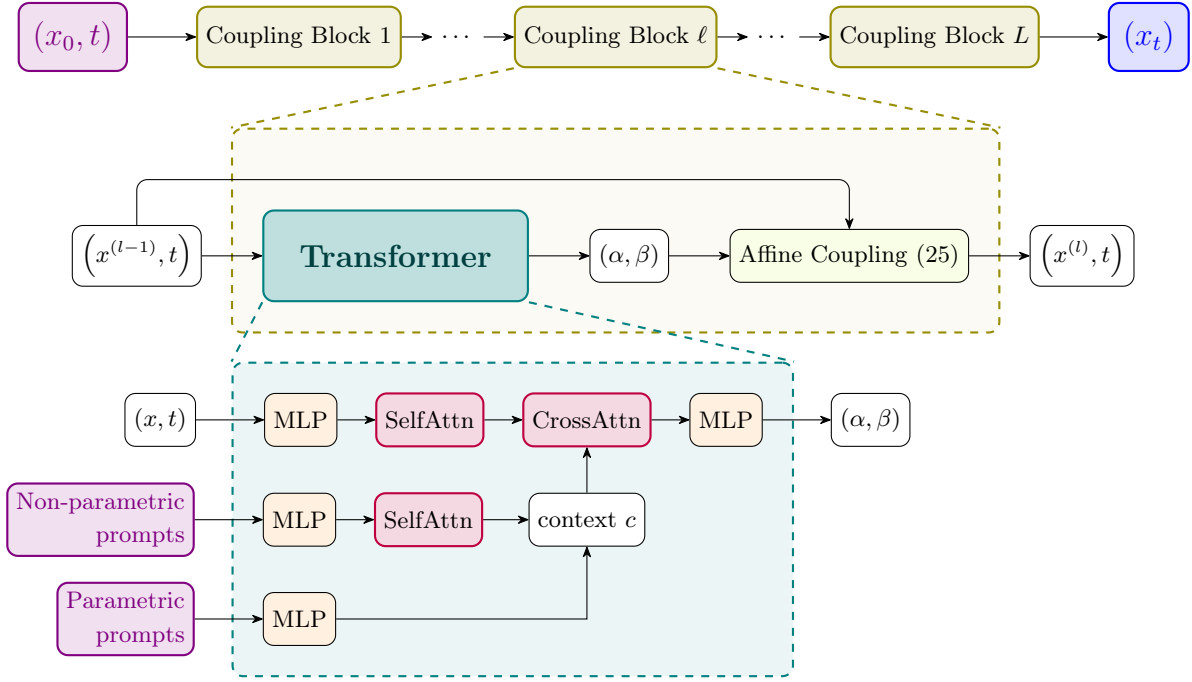
\begin{figure}[t]
\centering
\begin{tikzpicture}[>={Stealth[round]},font=\small,
  box/.style={fill=white,draw,rounded corners,minimum height=7mm,align=center},
  emb/.style={box,fill=blue!6}, 
  net/.style={box,fill=orange!12},
  input/.style={box,fill=violet!12,draw=violet,text=violet,thick},
  output/.style={box,fill=blue!12,draw=blue,text=blue,thick},
  attn/.style={box,fill=purple!15,draw=purple,thick}, 
  res/.style={box,fill=olive!10,minimum width=2.4cm,minimum height=8mm,thick,draw=olive}, 
  arr/.style={->}]
  
  \begin{scope}[node distance=4mm and 5mm]
      \node[net] (tmlp) {MLP};
      \node[attn,right=of tmlp] (sa) {SelfAttn}; 
      \node[attn,right=of sa] (ca) {CrossAttn};  
      \node[net,right=of ca] (mlp) {MLP};
      
      \foreach \a/\b in {tmlp/sa,sa/ca,ca/mlp} \draw[arr] (\a)--(\b);
      
      \node[box,below=6mm of ca] (c) {context $c$};
      
      \node[attn,left=6mm of c] (setattn) {SelfAttn};
      \node[net,left=of setattn] (embed1) {MLP};
      \node[input,left=9mm of embed1,align=right] (prompts) {Non-parametric\\prompts};

      \node[net,below=6mm of embed1] (embed2) {MLP};
      \node[input,left=9mm of embed2,align=right] (betas) {Parametric\\prompts};

      \draw[arr] (prompts)--(embed1);
      \draw[arr] (embed1)--(setattn);
      \draw[arr] (setattn)--(c.west);
      
      \draw[arr] (betas)--(embed2);
      \draw[arr] (embed2.east) -| (c.south); 
      
      \draw[arr] (c)--(ca); 
      
      \node[box, left=9mm of tmlp] (in_x0) {$(x,t)$};
      \node[box, right=9mm of mlp] (out_st) {$(\alpha,\beta)$};
      \draw[arr] (in_x0) -- (tmlp);
      \draw[arr] (mlp) -- (out_st);
  \end{scope}

  \begin{scope}[on background layer]
      \node[draw=teal, dashed, thick, fill=teal!8, rounded corners, inner sep=4mm, 
            fit= (tmlp) (mlp) (c) (embed2)] (magnified) {};
  \end{scope}

  \path (magnified.north) ++(0, 1.4) coordinate (mid_axis);

  \begin{scope}[node distance=6mm and 8mm]
      \node[box, fill=teal!25, minimum height=12mm, minimum width=3.5cm, anchor=center,thick,draw=teal,text=teal!50!black] (trans_mid) at ([xshift=-1.55cm]mid_axis) {\textbf{{\large Transformer}}};
      \node[box, left=of trans_mid] (q_mid) {$\left(x^{(l-1)},t\right)$};
      \node[box, right=of trans_mid] (st_mid) {$(\alpha,\beta)$};
      \node[box, fill=lime!10, right=of st_mid] (ell_mid) {Affine Coupling~\eqref{equ:coupling}};
      \node[box, right=of ell_mid] (q_mid_out) {$\left(x^{(l)},t\right)$};
      
      \draw[arr] (q_mid) -- (trans_mid);
      \draw[arr] (trans_mid) -- (st_mid);
      \draw[arr] (st_mid) -- (ell_mid);
      \draw[arr] (ell_mid) -- (q_mid_out);

      \draw[arr, rounded corners] (q_mid.north) -- ++(0, 0.6) -| (ell_mid.north);
  \end{scope}

  \begin{scope}[on background layer]
      \path (q_mid.north) ++(2.0, 0.75) node (jump_top) {}; 
      \node[draw=olive, dashed, thick, fill=olive!3, rounded corners, inner sep=4mm, 
            fit=(ell_mid) (trans_mid) (st_mid) (jump_top)] (mid_layer) {};
  \end{scope}

  \path (mid_layer.north) ++(0, 1.2) coordinate (macro_axis);

  \begin{scope}[node distance=6mm and 8mm]
      \node[res, anchor=center] (macro_ell) at (macro_axis) {Coupling Block $\ell$};
      \node[left=4mm of macro_ell] (macro_dots1) {$\dots$};
      \node[res, left=4mm of macro_dots1] (macro_1) {Coupling Block $1$};
      \node[input, left=9mm of macro_1,minimum height=9mm] (z_in) {{\large $\,(x_0,t)\,$}};

      \node[right=4mm of macro_ell] (macro_dots2) {$\dots$};
      \node[res, right=4mm of macro_dots2] (macro_L) {Coupling Block $L$};
      \node[output, right=9mm of macro_L,minimum height=9mm] (z_out) {{\large $\,(x_t)\,$}};

      \draw[arr] (z_in) -- (macro_1);
      \draw[arr] (macro_1) -- (macro_dots1);
      \draw[arr] (macro_dots1) -- (macro_ell);
      \draw[arr] (macro_ell) -- (macro_dots2);
      \draw[arr] (macro_dots2) -- (macro_L);
      \draw[arr] (macro_L) -- (z_out);
  \end{scope}
  
  \draw[dashed, olive, thick] (macro_ell.south west) -- (mid_layer.north west);
  \draw[dashed, olive, thick] (macro_ell.south east) -- (mid_layer.north east);

  \draw[dashed, teal, thick] (trans_mid.south west) -- (magnified.north west);
  \draw[dashed, teal, thick] (trans_mid.south east) -- (magnified.north east);

\end{tikzpicture}
\caption{
Architecture of NFIST.
\textbf{Top:} The solution operator is realized as a stack of affine coupling blocks.
\textbf{Middle:} Each coupling block uses a transformer to produce the affine parameters $(\alpha,\beta)$ that define the coupling transformation.
\textbf{Bottom:} The transformer conditions the flow on the task prompt. Query particles $(x,t)$ are processed by self-attention and cross-attention with a context $c$ constructed from either non-parametric particle-cloud prompts or parametric prompt tokens.
}
\label{fig:arch}
\end{figure}

\paragraph{Affine coupling layer}
Let $c$ denote the encoded task context. We construct a bijection
$F(\cdot,t,c)=G(\cdot,t)$ satisfying
$F(\cdot,0,c)=\mathrm{id}$,
with an explicit inverse and Jacobian log-determinant.

Let $A\subset\{1,\ldots,d\}$ be the conditioning coordinates,
$B=A^c$ the transformed coordinates, and
$\chi\in\{0,1\}^d$ the corresponding mask, where
$\chi_i=1$ iff $i\in A$. For any $x\in\R^d$, define the masked vectors $x_A:=\chi\odot x$ and $x_B:=(1-\chi)\odot x$, so that both belong to $\R^d$ and satisfy
$x=x_A+x_B$. A RealNVP affine coupling layer~\cite{dinh2017realnvp}
acts as
\begin{align}
\ell(x)=x_A+x_B\odot\exp(\alpha(x_A))+\beta(x_A),
\label{equ:coupling}
\end{align}
where $\alpha,\beta:\R^d\to\R^d$ vanish on the conditioning coordinates, so only the entries indexed by $B$ are modified.

If the coordinates are ordered as $(A,B)$, the Jacobian has the block form
\begin{align}
\nabla\ell(x) = \begin{pmatrix}
I & 0\\
* & D
\end{pmatrix},
\label{equ:coupling_jac}
\end{align}
where $D=\operatorname{diag}\!\left(
\exp(\alpha(x_A)_i)
\right)_{i\in B}$. Hence, $\log \bigl|\det \boldsymbol{\nabla}\ell(x)\bigr|=\sum_{i\in B} \alpha(x_A)_i$. The inverse is obtained explicitly as
\begin{align}
\ell^{-1}(y)
=
y_A
+
\bigl(y_B-\beta(y_A)\bigr)
\odot
\exp\bigl(-\alpha(y_A)\bigr),
\label{equ:coupling_inv}
\end{align}

\paragraph{Time-conditioned layer} 
The affine coupling layer defines a static bijection. To obtain a time-dependent, task-conditioned flow, we let the affine parameters depend on the masked input $x_A$, the time $t\in[0,T]$, and the context $c$. Each layer first predicts unconstrained fields $\tilde\alpha(x_A,t,c)$ and $\tilde\beta(x_A,t,c)$, which are transformed as
\begin{align}
    \alpha(x_A,t,c)
    &=t(1-\chi)\odot\tanh\bigl(\tilde\alpha(x_A,t,c)\bigr),
    \label{equ:alpha_scaled}\\
    \beta(x_A,t,c)
    &=t(1-\chi)\odot\tilde\beta(x_A,t,c).
    \label{equ:beta_scaled}
\end{align}
The time factor ensures $\alpha(\cdot,0,c)=\beta(\cdot,0,c)=0$, so every coupling layer, and hence the entire flow, is exactly the identity at $t=0$. The $\tanh$ bounds the log-scale $\alpha$, improving numerical stability, while the mask $(1-\chi)$ preserves the affine coupling structure by enforcing $\alpha_A=\beta_A=0$.

\paragraph{Prompt encoding and cross-attention conditioning}
Both task representations introduced in Section~\ref{sec:task_disc} are mapped to a common context representation. Each prompt slot is embedded by a token-wise MLP into a feature vector of dimension $d_{\mathrm{fea}}$, and the resulting tokens are concatenated into a context tensor $c\in\R^{N_c\times d_{\mathrm{fea}}}$, where $N_c$ is the total number of prompt tokens. Parametric inputs (Euclidean parameters or Gaussian tokens) contribute one token each, while a particle cloud contributes one token per sample. For particle-cloud prompts, a small self-attention encoder is first applied to capture interactions among the samples before forming the context.

Within each coupling layer, the affine parameters are generated from the masked input $x_A$ and the time $t$ through
\begin{align}
u &= \mathrm{TimeMLP}(x_A,t),\notag\\
u &\leftarrow \mathrm{SelfAttn}(u,u,u),\notag\\
u &\leftarrow \mathrm{CrossAttn}(u,c,c),\notag\\
(\tilde\alpha,\tilde\beta) &= \mathrm{MLP}_{\alpha\beta}(u),
\label{equ:layer_net}
\end{align}
where $\mathrm{MLP}_{\alpha\beta}$ outputs the unconstrained fields $\tilde\alpha,\tilde\beta\in\R^d$ used in
\eqref{equ:alpha_scaled}--\eqref{equ:beta_scaled}.

The self-attention layer captures interactions among the query particles, while the cross-attention layer conditions the transport on the task context $c$. Since all remaining operations are token-wise, the architecture is \emph{permutation-equivariant} in the query particles and \emph{permutation-invariant} in the prompt tokens, naturally accommodating variable numbers of particles and prompt tokens.

\paragraph{Stacked flow}
The complete flow is obtained by composing $L$ coupling layers with alternating masks,
\[
F(\cdot,t,c)=\ell_L\circ\cdots\circ\ell_1.
\]
Alternating the masks ensures that every coordinate is transformed across successive layers. Since each layer is the identity at $t=0$, the composition also satisfies $F(\cdot,0,c)=\mathrm{id}$. Let $z^{(0)}=x$ and $z^{(k)}=\ell_k(z^{(k-1)})$. By the chain rule, the Jacobian log-determinant is
\begin{align}
\log\bigl|\det\nabla_xF(x,t,c)\bigr|
=
\sum_{k=1}^{L}
\mathbf1^\top
\alpha^{(k)}\!\left(z^{(k-1)},t,c\right),
\label{equ:stack_logdet}
\end{align}
where $\alpha^{(k)}$ denotes the log-scale field of the $k$-th coupling layer.

At $d=1$, the construction specializes: one of the two blocks is necessarily empty, so the stack collapses to a single task- and time-conditioned affine map $x\mapsto e^{\bar\alpha(t,c)}x + \bar\beta(t,c)$. Its inverse and Jacobian log-determinant required for score evaluation remain available in explicit form, while the identity condition at $t=0$ continues to hold by construction.

\paragraph{Computational efficiency} 
Section~\ref{sec:task_disc} evaluates the score through one inverse pass of the flow followed by one reverse-mode sweep~\eqref{equ:score_discrete}. For the coupling-flow stack above, the inverse pass is explicit. Setting $\zeta^{(L)}=y$ and recursively $\zeta^{(k-1)}=\ell_k^{-1}(\zeta^{(k)})$, so that $\zeta^{(0)}=F^{-1}(y,t,c)$, the score at a transported point $y=F(x,t,c)$ is
\begin{align}
\mathbf{s}(y,t)
=
\nabla_y^{\mathrm{autodiff}}
\left[
\log p_0\bigl(F^{-1}(y,t,c)\bigr)
-
\sum_{k=1}^{L}
\mathbf1^\top
\alpha^{(k)}\!\left(\zeta^{(k-1)},t,c\right)
\right].
\label{equ:score_inverse}
\end{align}

At fixed depth and feature width, all flow-related operations scale linearly with the state dimension $d$. Each coupling layer has an explicit inverse through~\eqref{equ:coupling_inv}, and its log-determinant is obtained directly from the same scaling field $\alpha^{(k)}$ produced by the conditioner. Consequently, a forward or inverse pass requires one evaluation of each coupling layer, while accumulating the log-determinant adds only the elementwise sum of their $\mathcal O(d)$ scaling outputs. Moreover, the gradient in~\eqref{equ:score_inverse} is computed by a single reverse-mode sweep through the inverse pass. Within each coupling network~\eqref{equ:layer_net}, the dependence on $d$ appears only in the input lift $\R^{d+1}\to\R^h$, the output head $\R^h\to\R^{2d}$ producing $(\tilde\alpha,\tilde\beta)$, and the elementwise coupling operation~\eqref{equ:coupling}; all intermediate computations are performed at fixed feature width. Consequently, for fixed $L$ and network width, evaluating the flow map, its exact inverse and log-determinant, and the score all have $\mathcal O(d)$ complexity per particle.

Overall, the architecture satisfies the computational requirements of the training objective: the forward map, inverse map, Jacobian log-determinant, and score are all evaluated exactly and differentiated end-to-end, enabling efficient optimization of the expected objective~\eqref{equ:sol_op_obj}.

\section{Numerical Experiments}\label{sec:results}

We evaluate the proposed solution operator on four representative stochastic mean-field control problems: stochastic optimal control (Section~\ref{sec:exp_soc}), the Schr\"odinger bridge (Section~\ref{sec:exp_sb}), the systemic-risk model (Section~\ref{sec:exp_sr1}), and obstacle-avoiding path planning (Section~\ref{sec:exp_pp_task}).

Unless otherwise specified, all experiments use the same self-supervised training procedure followed by Algorithm~\ref{alg:soc}. At each iteration, tasks are sampled from the corresponding task meta distribution, the objective is evaluated using the analytical score formula of Lemma~\ref{lem:score_pushforward}, and the network parameters are updated by AdamW \cite{loshchilov2018decoupled}. Every operator uses the architecture of Section~\ref{sec:sol}, consisting of $8$ coupling layers with hidden width $64$, $4$ attention heads of width $64$, and GELU activations. Training uses a learning rate of $10^{-4}$, weight decay $10^{-3}$, gradient clipping at $1.0$, and a Fokker--Planck discretization with $n_{t}=16$ grid nodes, while inference uses $n_{\rm time}=32$ grid nodes. All experiments are trained and timed on a single NVIDIA RTX~4090 GPU (24 GB VRAM) with an Intel i9-9940X CPU.

\label{sec:exp_metrics}
Each experiment is evaluated against its exact reference (see details in \ref{sec:appendix_exact_sol}) using two relative errors, both averaged over the time horizon. For a task $\sT$, let $0=t_1<\cdots<t_{n_{\rm time}}=T$ be the inference grid, $x_j^i=G(x_0^i,t_j)$ the predicted trajectories, and $x_j^{\star,i}$ the corresponding reference trajectories generated from the same initial particles. Let $s=\nabla\log p$ denote the score function numerically computed by~\eqref{equ:score_inverse}, and let $s^\star$ be the corresponding reference score. 

\begin{align}
  E_{\rm traj}(\sT) &= \frac{1}{n_{\rm time}}\sum_{j=1}^{n_{\rm time}}
    \frac{\sum_{i=1}^{N_{\rm pt}}\norm{x^{i}_{j}-x^{\star,i}_{j}}}
         {\sum_{i=1}^{N_{\rm pt}}\norm{x^{\star,i}_{j}}},
    \label{equ:err_traj}\\
  E_{\rm score}(\sT) &= \frac{1}{n_{\rm time}}\sum_{j=1}^{n_{\rm time}}
    \frac{\sum_{i=1}^{N_{\rm pt}}\norm{s\qb{x^{i}_{j},t_j}-s^{\star}\qb{x^{i}_{j},t_j}}}
         {\sum_{i=1}^{N_{\rm pt}}\norm{s^{\star}\qb{x^{i}_{j},t_j}}}.
    \label{equ:err_score}
\end{align}

In all the tables, both timing columns (``transport" and ``$+$\,score") are measured on a single task with $N_{\rm pt}=100$ transported particles. The ``transport'' column is one batched rollout that evaluates the flow at all $n_{\rm time}$ nodes in a single call, whereas the ``$+$\,score'' column adds one score evaluation at every node. 

\subsection{Stochastic Optimal Control (SOC)}
\label{sec:exp_soc}

We first evaluate the proposed operator on stochastic optimal control problems of increasing complexity. Section~\ref{sec:exp_soc1} considers the single-well (quadratic-terminal) problem with parametric prompts, using the closed-form Gaussian solution discussed in \ref{sec:appendix_soc} as the exact reference across dimensions $d\in\{1,2,5,8,11,14\}$. Section~\ref{sec:exp_nonparam} replaces the parametric representation by non-parametric particle-cloud prompts to assess the effect of empirical conditioning. Finally, Section~\ref{sec:exp_soc_mw} considers a genuinely multimodal multi-well terminal potential, using the Hopf--Cole reference solution discussed in \ref{sec:appendix_soc_mw}.

\subsubsection{Single-well potential, parametric prompts}
\label{sec:exp_soc1}
\begin{figure}[!htpb]
    \centering
    \setlength{\abovecaptionskip}{3pt}
    \includegraphics[width=\linewidth]{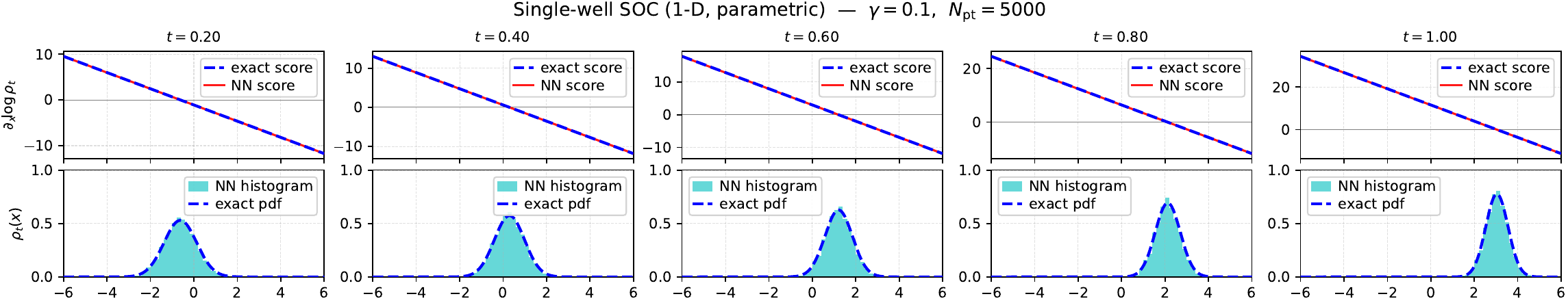}\\[3pt]
    \includegraphics[width=\linewidth]{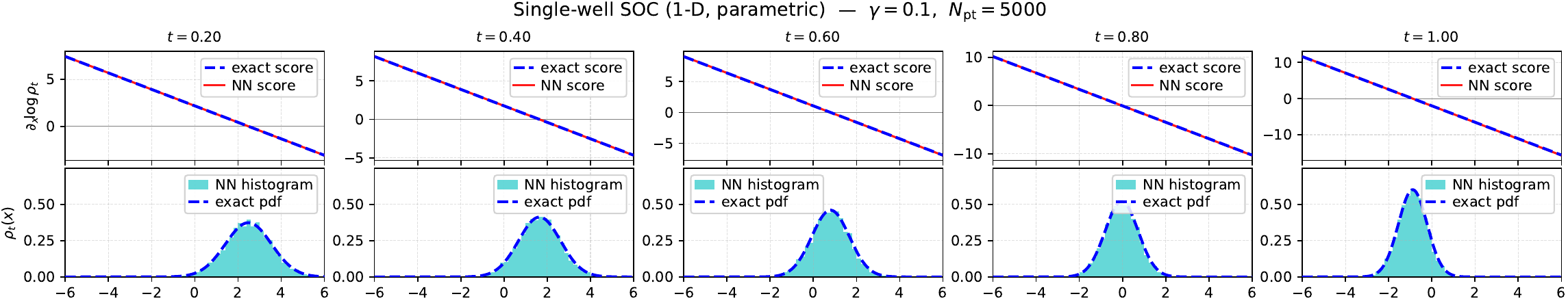}\\[3pt]
    \includegraphics[width=\linewidth]{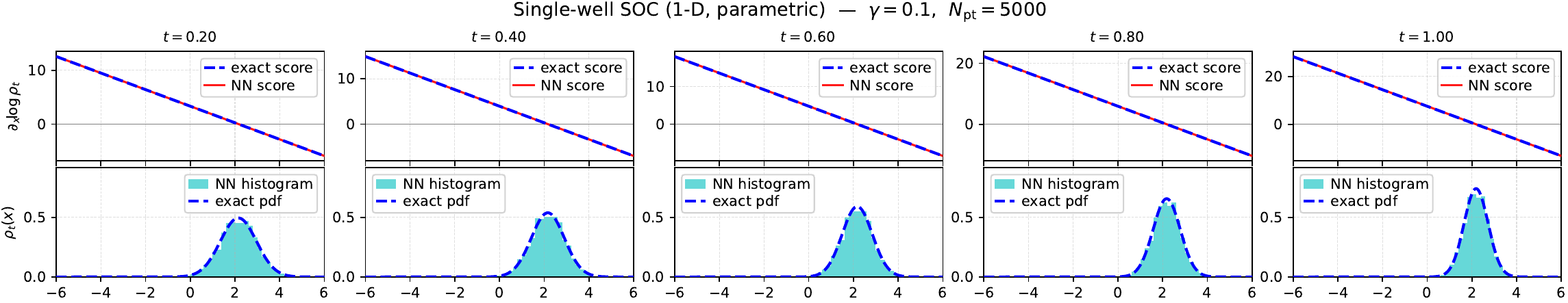}
    \caption{Single-well one-dimensional SOC problem, one solution operator netowrk ($\g=0.1$; random initial mean/variance and well location), on three tasks (strips). 
    In both rows the solid curve is the operator output and the dashed curve the
    exact reference, drawn on top, so exact agreement appears as a dashed curve
    riding on a solid one.}
    \label{fig:soc_density}
\end{figure}

We first consider the SOC problem with a quadratic terminal cost. Tasks are parameterized by the well center $\m_1$, sampled uniformly from a bounded region. In the one-dimensional experiments, the initial distribution $P_0=\sN(m_0,\sigma_0^2)$ is additionally randomized in both mean and variance, and the prompt contains $(m_0,\sigma_0,\m_1)$. The terminal potential is the squared distance to the well center $V(x)=\|x-\m_1\|^2$ so that the terminal cost $M(p(\cdot,T))=\int_\Omega V(x)\,p(x,T)\dd x$ penalizes deviation from $\m_1$. The closed form is available in \ref{sec:appendix_soc}.

Figure~\ref{fig:soc_density} compares the learned solution with the analytical Gaussian reference on three tasks. The predicted particle distribution accurately matches the exact density throughout the evolution, while the analytical score recovered from the learned flow remains nearly indistinguishable. Figure~\ref{fig:soc_sweep} illustrates this zero-shot generalization in two dimensions. The transport map and the predicted terminal distribution agree with the closed-form social optimum for each unseen task.

\begin{figure}[!htpb]
    \centering
    \includegraphics[width=\linewidth]{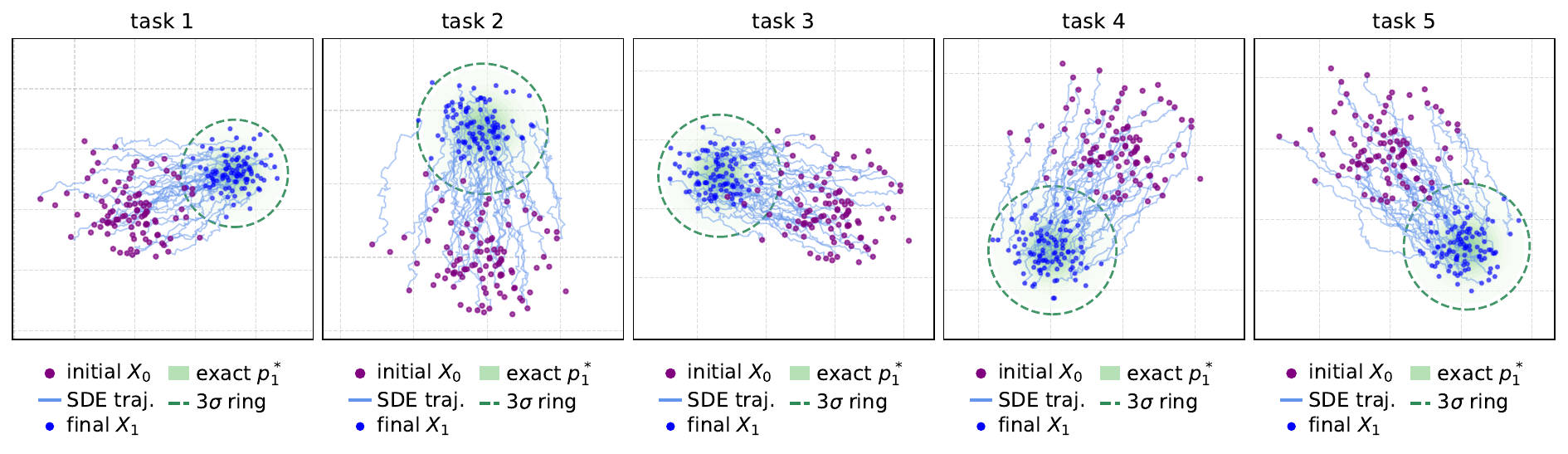}
    \caption{One SOC solution operator evaluated zero-shot over the plane of well
    centers  ($d=2$, $\gamma=0.1$). The SDE trajectories are simulated according to section~\ref{sec:sde_recovery}.
    }
    \label{fig:soc_sweep}
\end{figure}

Table~\ref{tab:soc_error_comparison} reports quantitative errors across diffusion coefficients and state dimensions. For $d=2$, the relative trajectory error remains at the $10^{-3}$ level over two orders of magnitude in the diffusion coefficient $\gamma$, indicating that the analytical score accurately captures the Fokker--Planck dynamics across diffusion regimes. The one-dimensional experiments consider a richer task family by additionally varying the initial mean and variance, yet achieve comparable accuracy. To assess scalability, we further train one operator for each dimension $d\in\{5,8,11,14\}$. Both trajectory and score errors remain at the $10^{-3}$ level, demonstrating that the proposed operator maintains high accuracy in high dimensions.

\begin{table}[!htpb]
    \centering
    \caption{Single-well SOC vs.\ the closed-form Gaussian reference: per-task
    inference time and the relative errors~\eqref{equ:err_traj}
    and~\eqref{equ:err_score}. The $d=1$ family additionally
    randomizes the initial mean and variance.
    }
    \label{tab:soc_error_comparison}
    \begin{tabular}{cccccc}
        \toprule
         &  & \multicolumn{2}{c}{\textbf{Inference Time (s)}}
            & \multicolumn{2}{c}{\textbf{Relative error}} \\
        \cmidrule(lr){3-4}\cmidrule(lr){5-6}
        \textbf{dim} & $\boldsymbol{\g}$ & transport & $+$\,score
            & trajectory & score \\
        \midrule
        $1$  & $0.01$ & 0.013 & 0.752 & 3.60e-4 $\pm$ 2.78e-4 & 1.11e-3 $\pm$ 7.99e-4 \\
        $1$  & $0.1$  & 0.013 & 1.066 & 1.11e-3 $\pm$ 6.39e-4 & 3.29e-3 $\pm$ 1.74e-3 \\
        $1$  & $1.0$  & 0.013 & 0.650 & 2.57e-3 $\pm$ 1.51e-3 & 6.84e-3 $\pm$ 3.40e-3 \\
        \midrule
        $2$  & $0.01$ & 0.012 & 0.486 & 2.45e-3 $\pm$ 1.11e-3 & 9.67e-3 $\pm$ 4.13e-3 \\
        $2$  & $0.1$  & 0.010 & 0.539 & 1.07e-3 $\pm$ 4.19e-4 & 4.83e-3 $\pm$ 1.50e-3 \\
        $2$  & $1.0$  & 0.011 & 1.049 & 3.06e-3 $\pm$ 1.18e-3 & 8.83e-3 $\pm$ 3.90e-3 \\
        \midrule
        $5$  & $1.0$  & 0.013 & 1.363 & 2.75e-3 $\pm$ 7.47e-4 & 6.28e-3 $\pm$ 1.48e-3 \\
        $8$  & $1.0$  & 0.012 & 1.354 & 4.27e-3 $\pm$ 9.47e-4 & 8.48e-3 $\pm$ 1.70e-3 \\
        $11$ & $1.0$  & 0.012 & 0.677 & 4.31e-3 $\pm$ 8.40e-4 & 8.23e-3 $\pm$ 1.41e-3 \\
        $14$ & $1.0$  & 0.013 & 1.257 & 4.06e-3 $\pm$ 7.82e-4 & 7.07e-3 $\pm$ 1.24e-3 \\
        \bottomrule
    \end{tabular}
\end{table}

\subsubsection{Single-well potential, non-parametric prompts}
\label{sec:exp_nonparam}

\begin{figure}[htpb]
    \centering
    \setlength{\abovecaptionskip}{3pt}
    \includegraphics[width=\linewidth]{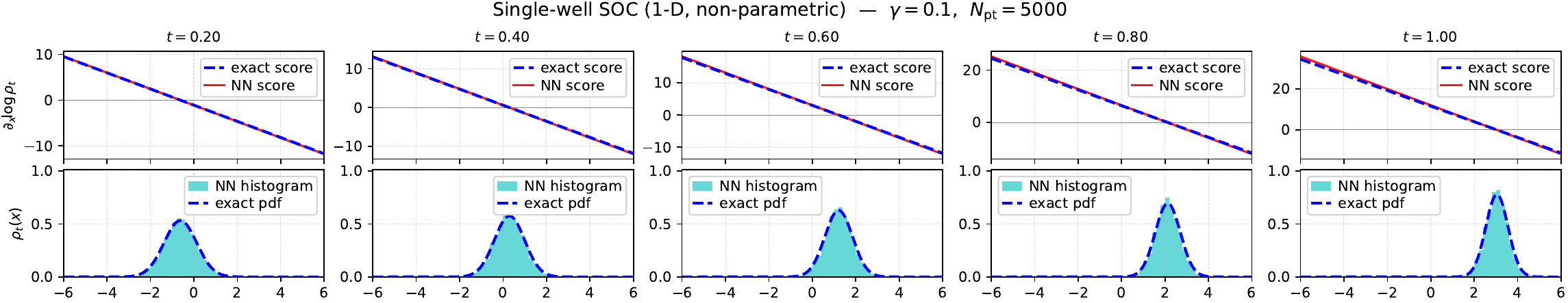}\\[3pt]
    \includegraphics[width=\linewidth]{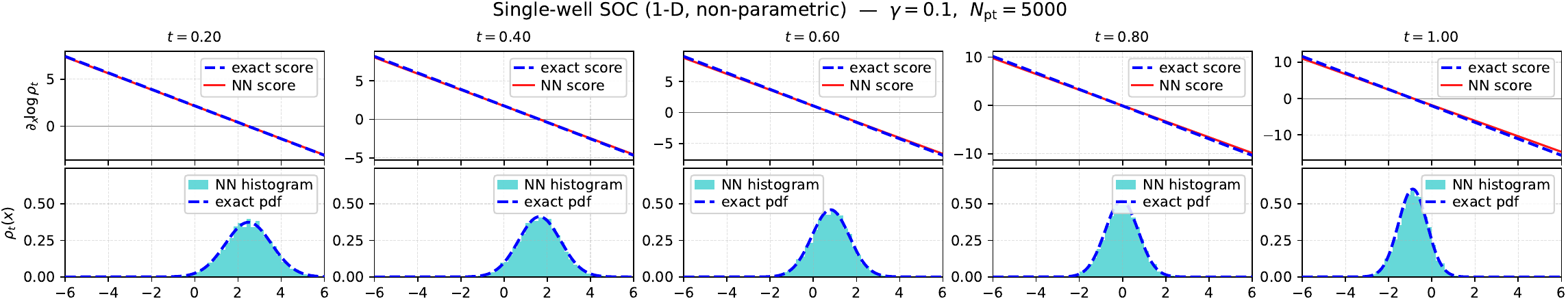}\\[3pt]
    \includegraphics[width=\linewidth]{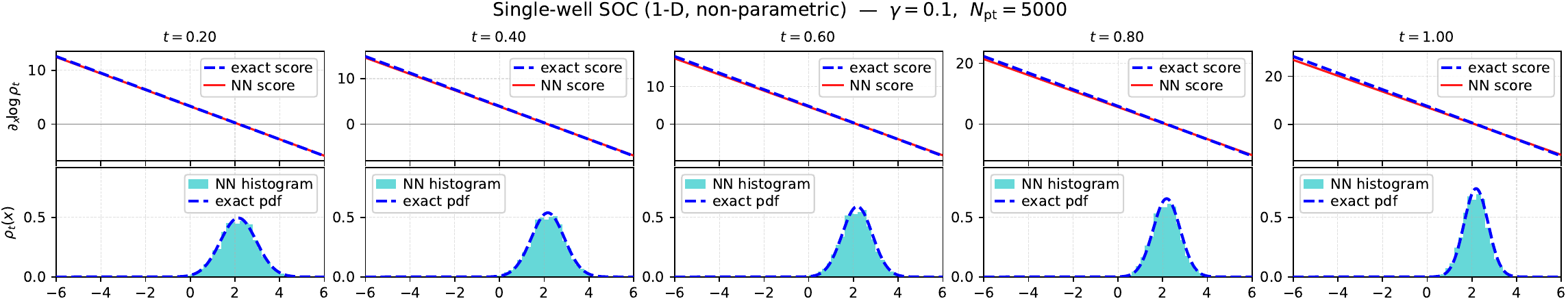}
    \caption{Single-well SOC with the \emph{non-parametric} particle-cloud
    prompt ($d=1$, $\g=0.1$): the operator is conditioned on raw $100$-point
    sample clouds $X_0$ rather than on $(m_0,\sigma_0)$. Same three 
    tasks and panel layout as the parametric Figure~\ref{fig:soc_density}.}
    \label{fig:isoc_fields}
\end{figure}

We next replace the parametric prompt by the non-parametric particle-cloud prompt of Section~\ref{sec:task_disc}, in which each task is represented by two sets of $N_{\rm pt}=100$ i.i.d. samples $X_0\sim P_0$, and $X_V\sim e^{-V(x)}$, representing the potential well $V(x)$ non-parametrically. All experimental settings, evaluation metrics, and task families are identical to those in Section~\ref{sec:exp_soc1}; only the prompt representation is changed, allowing a direct comparison between parametric and empirical conditioning.

\begin{figure}[htpb]
    \centering
    \includegraphics[width=\linewidth]{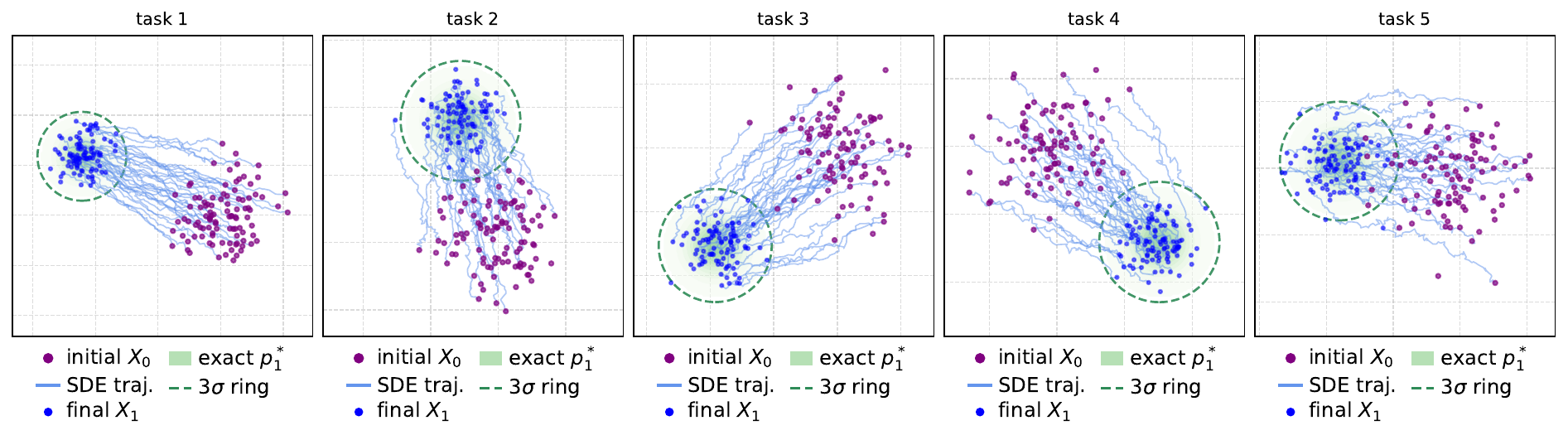}
    \caption{One non-parametric SOC solution operator evaluated zero-shot on 
    sampled tasks ($d=2$, $\gamma=0.1$). 
    The SDE trajectories are simulated according to section~\ref{sec:sde_recovery}.}
    \label{fig:isoc_sweep}
\end{figure}

Figures~\ref{fig:isoc_fields} and~\ref{fig:isoc_sweep} show that the learned operator remains qualitatively consistent with the parametric case. The predicted particle distribution closely matches the analytical Gaussian solution, the recovered score remains accurate throughout the evolution, and a single operator generalizes to unseen well locations under particle-cloud prompts.

Table~\ref{tab:nonparam_soc_error} reports the quantitative comparison. Relative trajectory errors and relative score errors remain at the $10^{-2}$ level, which is larger than those of the parametric operator. This loss of accuracy reflects the additional difficulty of learning directly from empirical particle clouds. The performance also remains stable as the state dimension increases. 

\begin{table}[!htpb]
    \centering
    \caption{Non-parametric (particle-cloud) single-well SOC vs.\ the exact
    Gaussian reference.
    }
    \label{tab:nonparam_soc_error}
    \begin{tabular}{cccccc}
        \toprule
         &  & \multicolumn{2}{c}{\textbf{Inference Time (s)}}
            & \multicolumn{2}{c}{\textbf{Relative error}} \\
        \cmidrule(lr){3-4}\cmidrule(lr){5-6}
        \textbf{dim} & $\boldsymbol{\g}$ & transport & $+$\,score
            & trajectory & score \\
        \midrule
        $1$  & $0.01$ & 0.014 & 1.370 & 1.73e-2 $\pm$ 1.66e-2 & 4.87e-2 $\pm$ 4.59e-2 \\
        $1$  & $0.1$  & 0.014 & 0.651 & 1.88e-2 $\pm$ 1.66e-2 & 4.93e-2 $\pm$ 4.00e-2 \\
        $1$  & $1.0$  & 0.014 & 1.373 & 3.17e-2 $\pm$ 2.26e-2 & 8.25e-2 $\pm$ 5.72e-2 \\
        \midrule
        $2$  & $0.01$ & 0.014 & 1.346 & 1.36e-2 $\pm$ 8.21e-3 & 4.72e-2 $\pm$ 2.42e-2 \\
        $2$  & $0.1$  & 0.014 & 0.664 & 1.24e-2 $\pm$ 7.47e-3 & 4.00e-2 $\pm$ 2.02e-2 \\
        $2$  & $1.0$  & 0.014 & 1.348 & 1.24e-2 $\pm$ 7.56e-3 & 2.73e-2 $\pm$ 1.47e-2 \\
        \midrule
        $5$  & $1.0$  & 0.013 & 1.322 & 1.63e-2 $\pm$ 5.59e-3 & 3.09e-2 $\pm$ 1.07e-2 \\
        $8$  & $1.0$  & 0.014 & 0.661 & 2.04e-2 $\pm$ 5.13e-3 & 3.54e-2 $\pm$ 8.47e-3 \\
        $11$ & $1.0$  & 0.014 & 1.359 & 2.10e-2 $\pm$ 4.57e-3 & 3.35e-2 $\pm$ 6.98e-3 \\
        $14$ & $1.0$  & 0.014 & 1.348 & 2.42e-2 $\pm$ 5.37e-3 & 3.69e-2 $\pm$ 7.93e-3 \\
        \bottomrule
    \end{tabular}
\end{table}

\subsubsection{Multi-well potential}
\label{sec:exp_soc_mw}

\begin{figure}[!htpb]
    \centering
    \includegraphics[width=\linewidth]{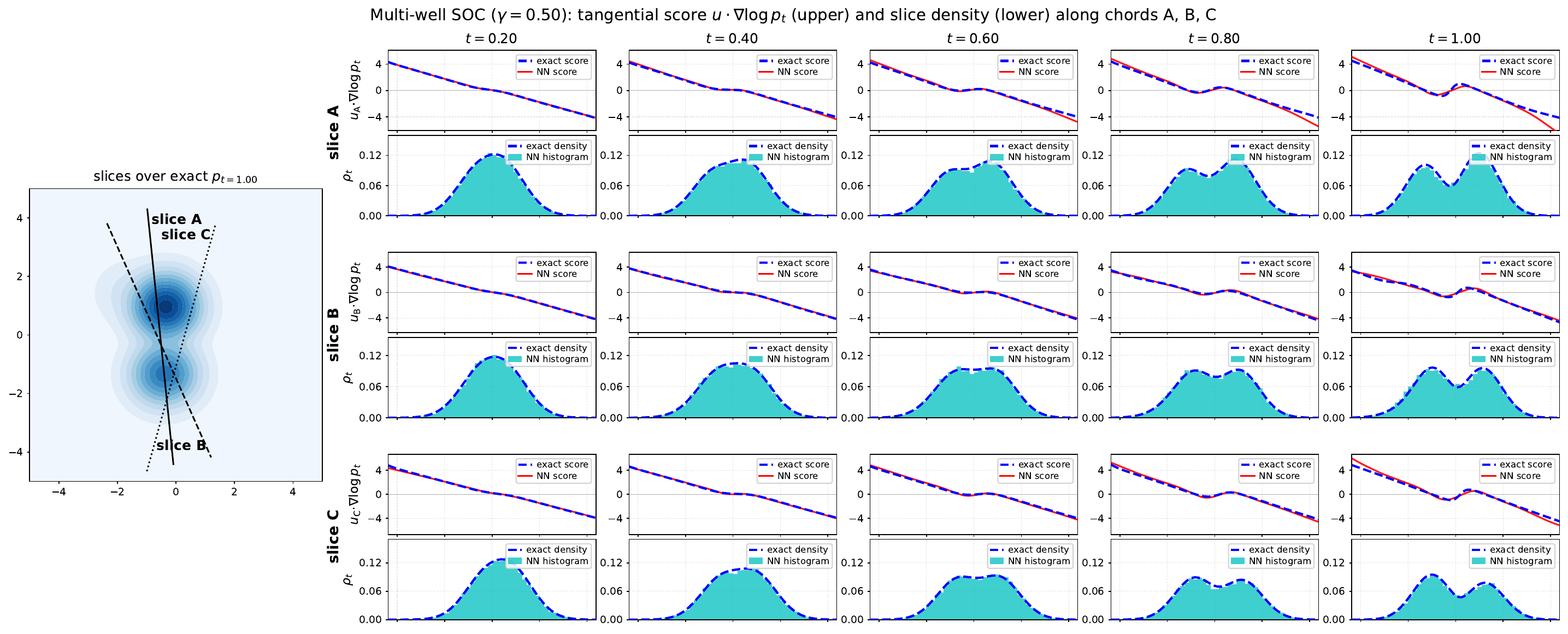}\\[12pt]
    \includegraphics[width=\linewidth]{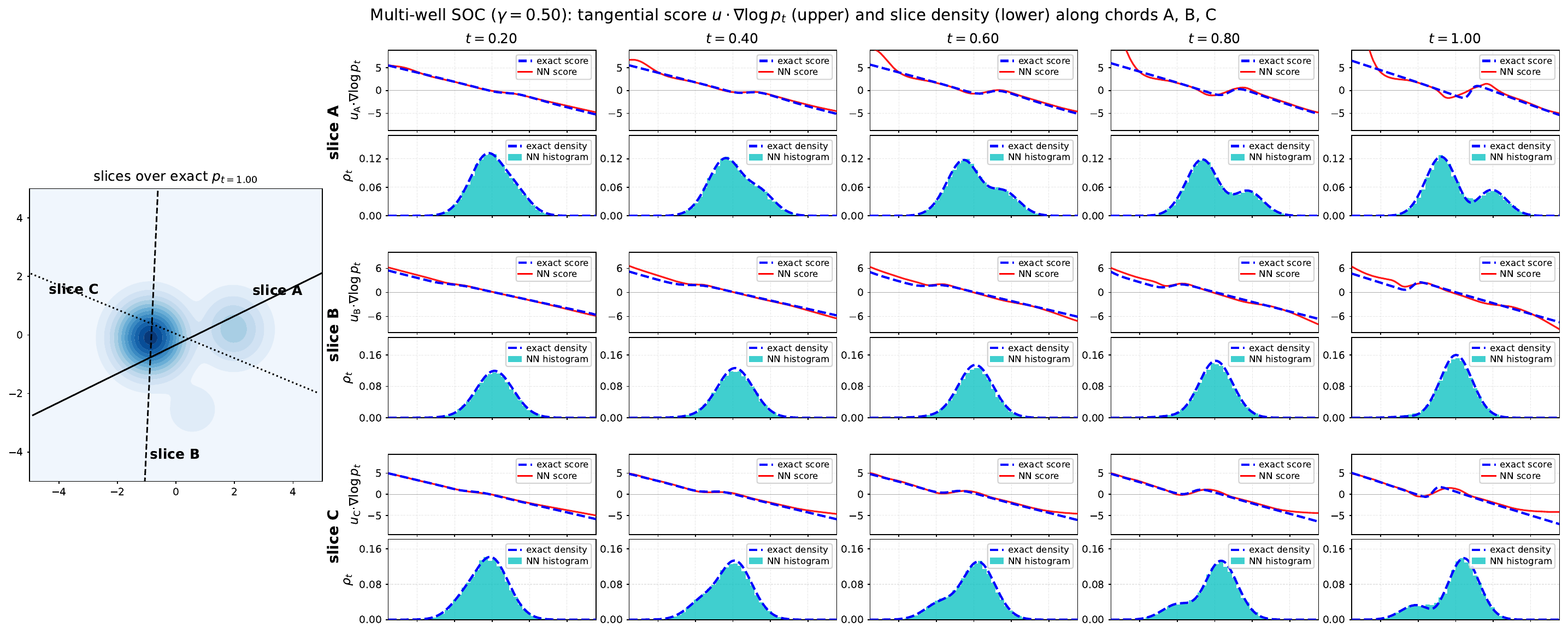}\\[12pt]
    \includegraphics[width=\linewidth]{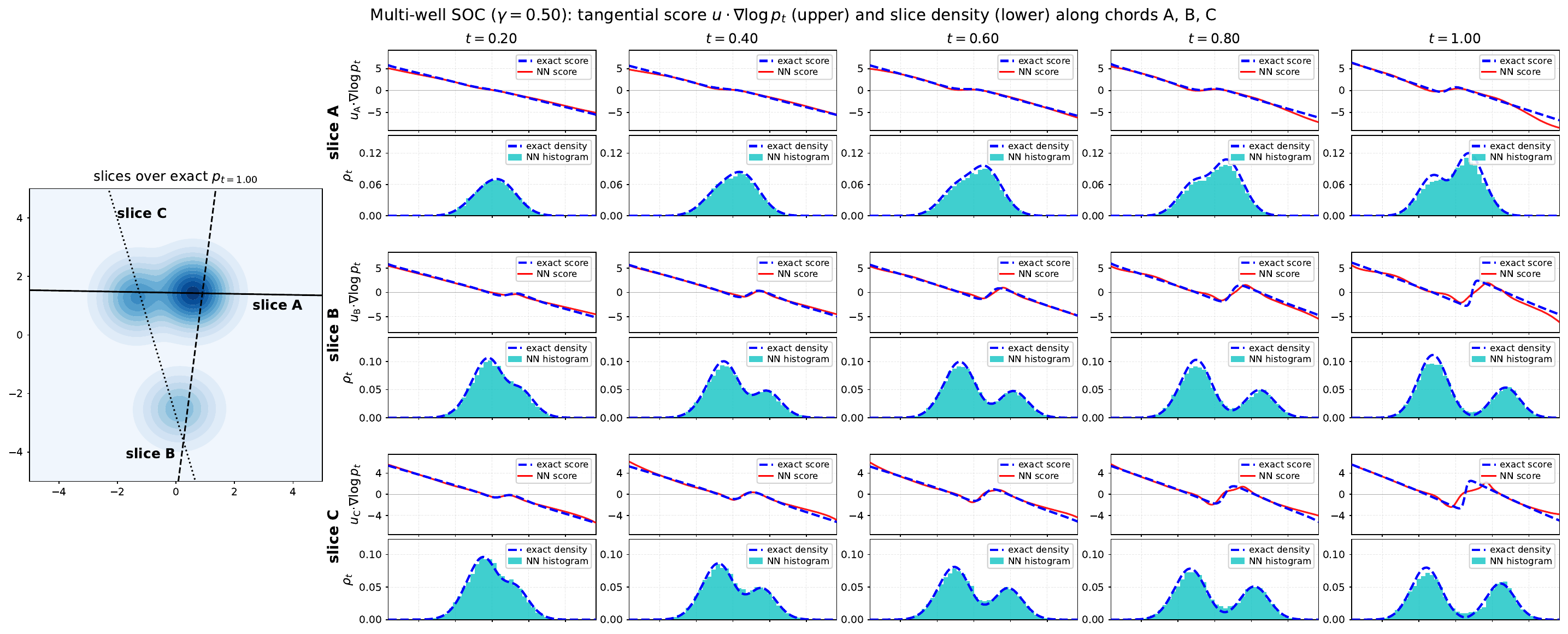}
    \caption{Multi-well SOC in the sharply multimodal $\g=0.5$ regime of three tasks. 
    Each strip samples the score along three high-density slices of the terminal law (as in the 2-D plot on the left).}
    \label{fig:mwsoc_g05_fields1d}
\end{figure}

We next consider a multimodal SOC problem with a three-well terminal potential, 
\[
    V(x)=-\log\sum_{i=1}^{n_V}e^{-\|x-\m_i\|^2}
    \ ,
\]
where the well centers $\m_i$ are sampled uniformly on a circle of radius $4$. Unlike the single-well problem, the initial Gaussian distribution must split and move toward multiple wells, producing a multimodal terminal density. The analytical Gaussian solution is no longer available, so the reference solution is computed from the Hopf--Cole formulation (see \ref{sec:appendix_soc_mw}). 

To increase the complexity of the multimodal target together with the state dimension, we consider $d\in\{2,5,8,11,14\}$ and, for each task, sample $n_V=d+1$ well centers subject to guarantee pairwise separation. For each dimension, we train one parametric and one non-parametric operator using matched batch sizes and evaluate the error metrics \eqref{equ:err_traj} and \eqref{equ:err_score}. The high-dimensional reference solutions are obtained by Monte Carlo evaluation of the Hopf--Cole representation described in Section~\ref{sec:appendix_soc_mw}; in $d=2$, this reference is further validated against the grid-based solver. Figure~\ref{fig:mwsoc_g05_fields1d} illustrates a representative two-dimensional task, showing that the learned operator accurately captures both the multimodal density evolution and the corresponding score score function in the high-probability region. The quantitative results in Table~\ref{tab:mwsoc_error} further show that both prompt representations remain effective as the state dimension and the number of potential wells increase.

\begin{table}[!htpb]
    \centering
    \caption{Multi-well SOC with wells spanning $\R^d$ vs.\ the Monte-Carlo
    Hopf--Cole reference of \ref{sec:appendix_soc_mw}: per-task inference time
    and the relative errors~\eqref{equ:err_traj} and~\eqref{equ:err_score}, the
    trajectory scored against the reference probability
    flow~\eqref{equ:hopfcole_pfode}. }
    \label{tab:mwsoc_error}
    \begin{tabular}{cccccc}
        \toprule
         &  & \multicolumn{2}{c}{\textbf{Inference Time (s)}}
            & \multicolumn{2}{c}{\textbf{Relative error}} \\
        \cmidrule(lr){3-4}\cmidrule(lr){5-6}
        \textbf{prompt type} & \textbf{dim} & transport & $+$\,score
            & trajectory & score \\
        \midrule
        Parametric & $2$ & 0.012 & 1.331 & 2.16e-2 $\pm$ 4.32e-3 & 7.25e-2 $\pm$ 1.23e-2 \\
        Parametric & $5$ & 0.012 & 1.330 & 4.29e-2 $\pm$ 3.00e-3 & 1.10e-1 $\pm$ 7.47e-3 \\
        Parametric & $8$ & 0.012 & 1.364 & 4.44e-2 $\pm$ 2.51e-3 & 1.07e-1 $\pm$ 4.58e-3 \\
        Parametric & $11$ & 0.012 & 1.380 & 4.56e-2 $\pm$ 1.61e-3 & 1.02e-1 $\pm$ 2.78e-3 \\
        Parametric & $14$ & 0.012 & 1.311 & 4.48e-2 $\pm$ 1.97e-3 & 9.57e-2 $\pm$ 3.05e-3 \\
        \midrule
        Non-parametric & $2$ & 0.014 & 1.335 & 2.71e-2 $\pm$ 6.38e-3 & 7.49e-2 $\pm$ 1.59e-2 \\
        Non-parametric & $5$ & 0.014 & 1.343 & 5.35e-2 $\pm$ 5.21e-3 & 1.32e-1 $\pm$ 1.10e-2 \\
        Non-parametric & $8$ & 0.013 & 1.346 & 6.02e-2 $\pm$ 3.33e-3 & 1.41e-1 $\pm$ 6.49e-3 \\
        Non-parametric & $11$ & 0.014 & 1.343 & 6.16e-2 $\pm$ 2.55e-3 & 1.41e-1 $\pm$ 5.90e-3 \\
        Non-parametric & $14$ & 0.013 & 1.340 & 6.30e-2 $\pm$ 2.54e-3 & 1.41e-1 $\pm$ 4.91e-3 \\
        \bottomrule
    \end{tabular}
\end{table}

\subsection{Schr\"odinger Bridge}
\label{sec:exp_sb}

We next evaluate the proposed operator on the Schr\"odinger bridge problem, which transports a prescribed initial distribution to a prescribed terminal distribution under entropy-regularized diffusion. The closed-form Gaussian bridge in ~\ref{sec:appendix_sb} provides the exact reference solution. Section~\ref{sec:exp_sb_gauss} considers parametric prompts for families of Gaussian endpoint pairs across dimensions $d\in\{1,2,5,8,11,14\}$, while Section~\ref{sec:exp_nonparam_sb} presents the results with non-parametric particle-cloud prompts. Compared with the stochastic optimal control problem, prescribing both endpoint distributions makes the bridge problem more challenging.

\subsubsection{Gaussian endpoints, parametric prompts}
\label{sec:exp_sb_gauss}

\begin{figure}[htpb]
    \centering
    \setlength{\abovecaptionskip}{3pt}
    \includegraphics[width=\linewidth]{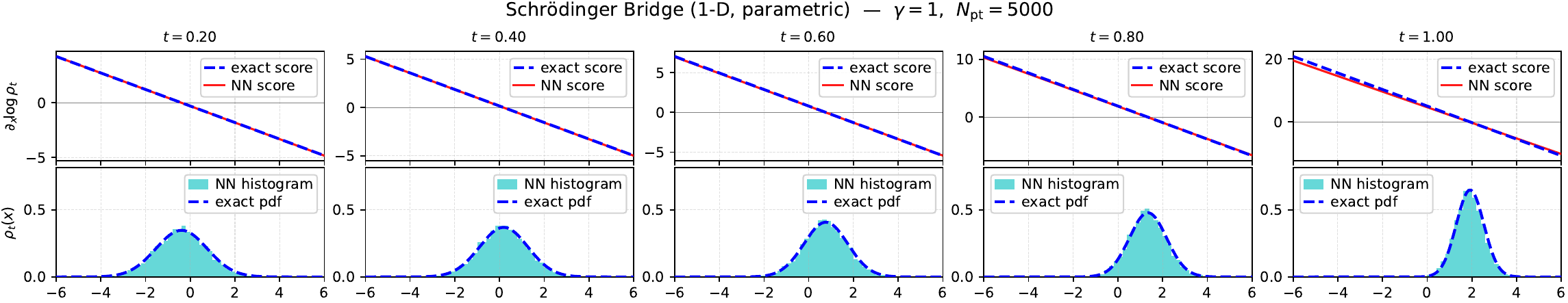}\\[3pt]
    \includegraphics[width=\linewidth]{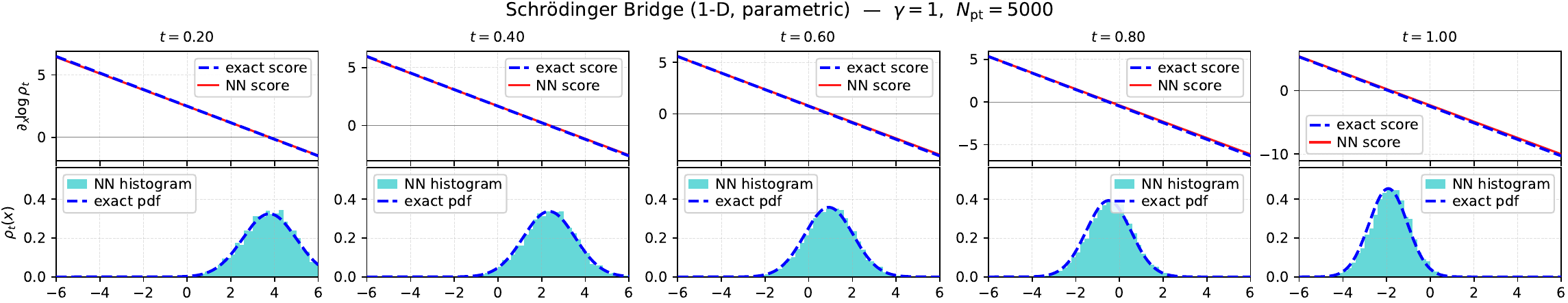}\\[3pt]
    \includegraphics[width=\linewidth]{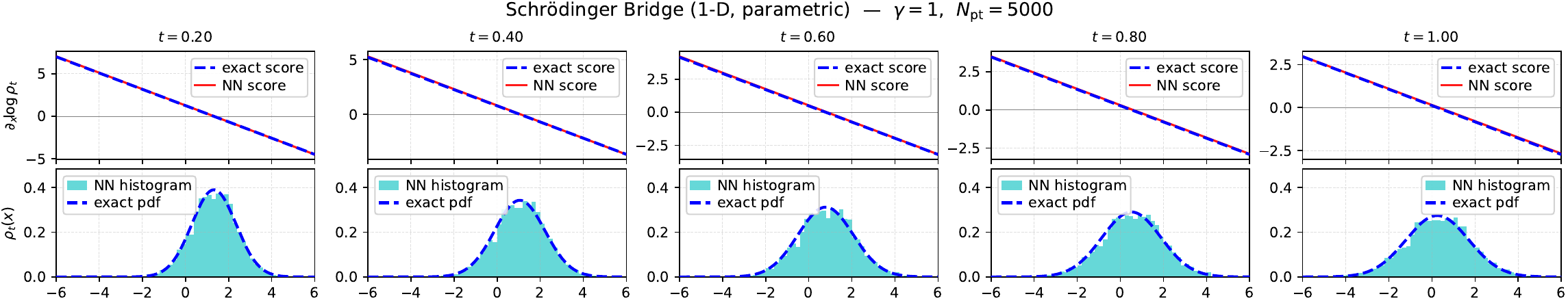}
    \caption{Schr\"odinger bridge, one-dimensional parametric solution operator
    ($\g=1.0$; random mean and variance at both endpoints).}
    \label{fig:sb_density}
\end{figure}

Figure~\ref{fig:sb_density} compares the learned bridge with the analytical Gaussian solution on three tasks. The predicted particle distribution closely matches the exact marginals throughout the evolution, while the recovered score remains in good agreement with the analytical score. Figure~\ref{fig:sb_sweep} further demonstrates zero-shot generalization: a single operator transports different initial distributions to their prescribed terminal distributions. Table~\ref{tab:sb_error_comparison} reports quantitative results across dimensions $d\in\{1,2,5,8,11,14\}$. Both relative trajectory errors and relative score errors remain at the $10^{-2}$ level, which is robust against varying state dimensions and the diffusion coefficient. 

\begin{figure}[!htpb]
    \centering
    \includegraphics[width=\linewidth]{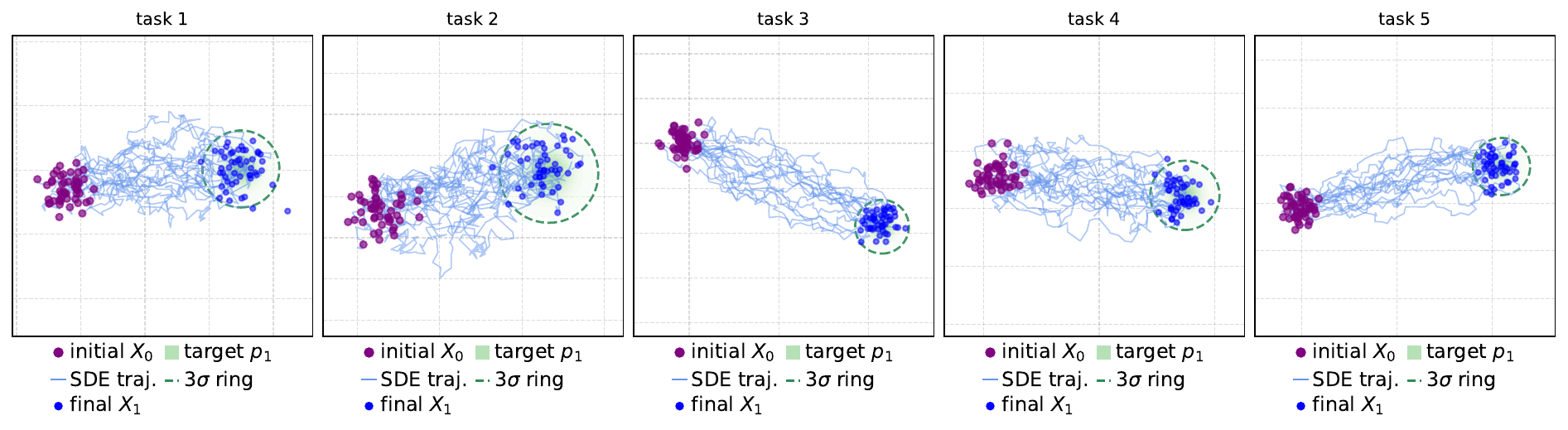}
    \caption{One bridge solution operator ($d=2$, $\g=5.0$), solved zero-shot; each panel framed to its own task. The prescribed terminal law $P_1$
    (green shading, dashed $3\sigma$) is the bridge's exact terminal marginal;
    trajectories (blue) carry each initial population (purple) onto it under
    fixed weights.
    The SDE trajectories are simulated according to section~\ref{sec:sde_recovery}.}
    \label{fig:sb_sweep}
\end{figure}

\begin{table}[!htpb]
    \centering
    \caption{Schr\"odinger bridge vs.\ the closed-form Gaussian reference:
    per-task inference time and the relative errors~\eqref{equ:err_traj}
    and~\eqref{equ:err_score}, The $d=1$ family randomizes mean and
    variance at both endpoints. 
    }
    \label{tab:sb_error_comparison}
    \begin{tabular}{cccccc}
        \toprule
         &  & \multicolumn{2}{c}{\textbf{Inference Time (s)}}
            & \multicolumn{2}{c}{\textbf{Relative error}} \\
        \cmidrule(lr){3-4}\cmidrule(lr){5-6}
        \textbf{dim} & $\boldsymbol{\g}$ & transport & $+$\,score
            & trajectory & score \\
        \midrule
        $1$  & $1.0$  & 0.014 & 0.714 & 3.27e-2 $\pm$ 2.47e-2 & 6.57e-2 $\pm$ 4.67e-2 \\
        $1$  & $5.0$  & 0.013 & 1.362 & 3.13e-2 $\pm$ 1.83e-2 & 6.90e-2 $\pm$ 2.93e-2 \\
        \midrule
        $2$  & $1.0$  & 0.012 & 1.379 & 2.12e-2 $\pm$ 1.16e-2 & 6.83e-2 $\pm$ 3.42e-2 \\
        $2$  & $5.0$  & 0.012 & 1.316 & 2.16e-2 $\pm$ 1.04e-2 & 6.14e-2 $\pm$ 2.45e-2 \\
        \midrule
        $5$  & $1.0$  & 0.012 & 1.322 & 2.28e-2 $\pm$ 7.57e-3 & 6.15e-2 $\pm$ 1.92e-2 \\
        $8$  & $1.0$  & 0.013 & 1.332 & 2.25e-2 $\pm$ 6.18e-3 & 5.38e-2 $\pm$ 1.28e-2 \\
        $11$ & $1.0$  & 0.012 & 1.174 & 2.58e-2 $\pm$ 7.41e-3 & 5.55e-2 $\pm$ 1.37e-2 \\
        $14$ & $1.0$  & 0.012 & 1.328 & 2.77e-2 $\pm$ 6.46e-3 & 5.55e-2 $\pm$ 1.11e-2 \\
        \bottomrule
    \end{tabular}
\end{table}

\subsubsection{Gaussian endpoints, non-parametric prompts}
\label{sec:exp_nonparam_sb}
We next consider non-parametric conditioning, in which the boundary distributions are represented directly by particle-cloud prompts $X_0\sim P_0$ and $X_1\sim P_1$. We keep the experimental settings, task families, and evaluation metrics identical to those in Section~\ref{sec:exp_sb_gauss}, changing only the prompt representation.

Figures~\ref{fig:isb_fields} and~\ref{fig:isb_sweep} show that the learned bridge remains qualitatively consistent with the parametric case. The predicted particle distribution accurately matches the analytical Gaussian marginals throughout the evolution, the recovered score remains close to the analytical score, and a single operator generalizes to unseen endpoint pairs under particle-cloud prompts.
Table~\ref{tab:nonparam_sb_error} reports the quantitative comparison. Relative trajectory errors remain at the $10^{-2}$ level and relative score errors at the $10^{-1}$ level across all dimensions and diffusion coefficients. Compared with the corresponding parametric operators, empirical particle-cloud prompts approximately doubled both trajectory and score errors, with little dependence on either the state dimension or the diffusion coefficient. 

\begin{figure}[htpb]
    \centering
    \setlength{\abovecaptionskip}{3pt}
    \includegraphics[width=\linewidth]{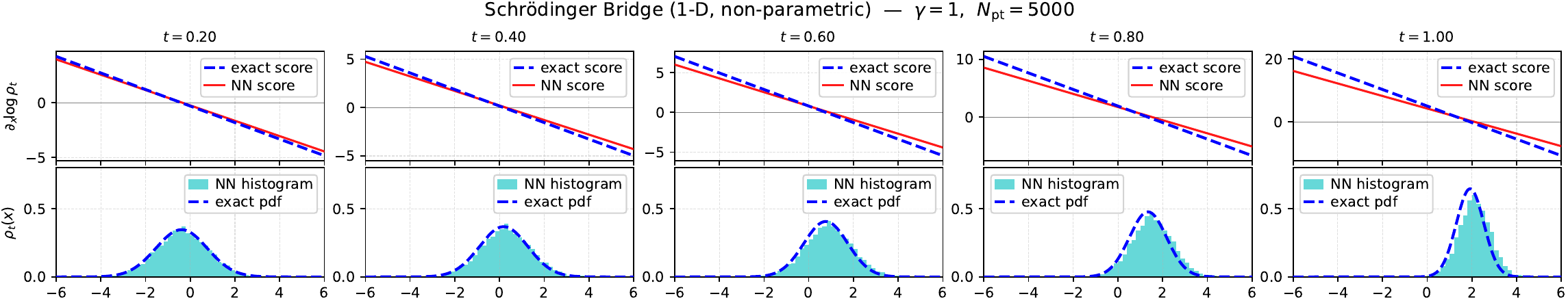}\\[3pt]
    \includegraphics[width=\linewidth]{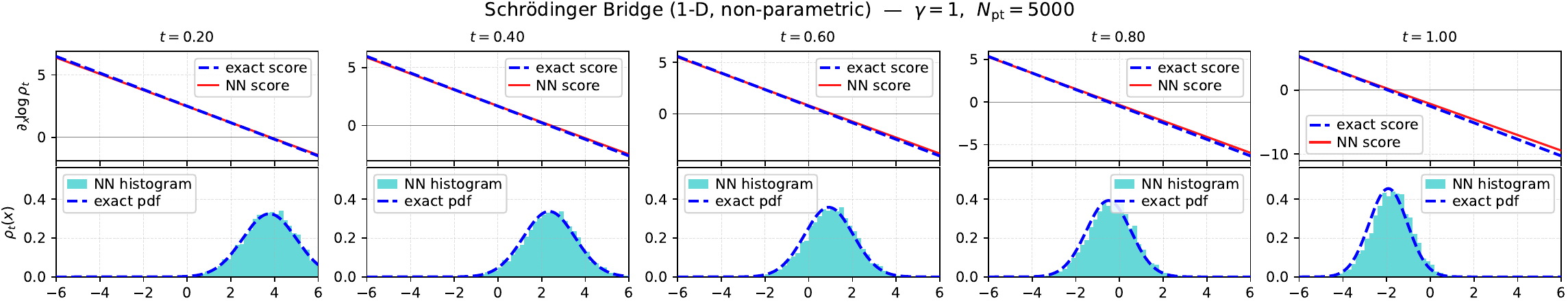}\\[3pt]
    \includegraphics[width=\linewidth]{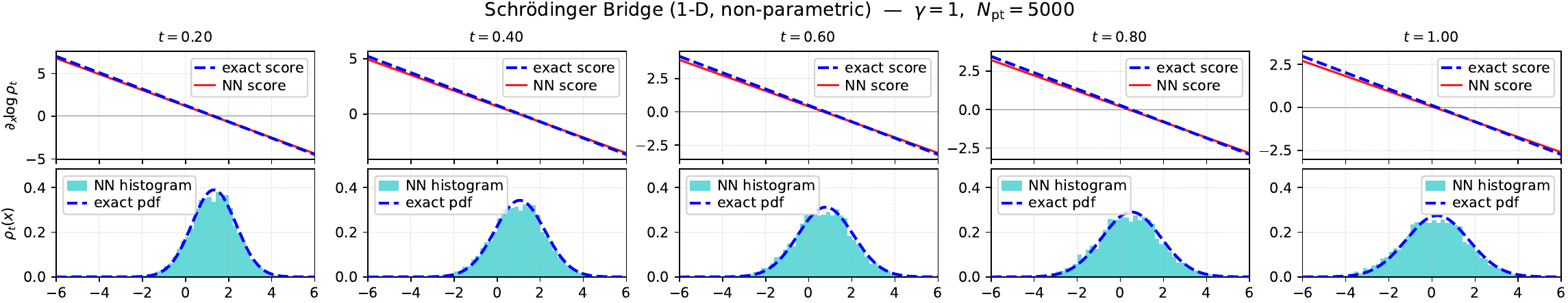}
    \caption{Schr\"odinger bridge with the \emph{non-parametric} particle-cloud
    prompt ($d=1$, $\g=1.0$; endpoints prompted by $100$-point sample clouds).
    }
    \label{fig:isb_fields}
\end{figure}

\begin{figure}[htpb]
    \centering
    \includegraphics[width=\linewidth]{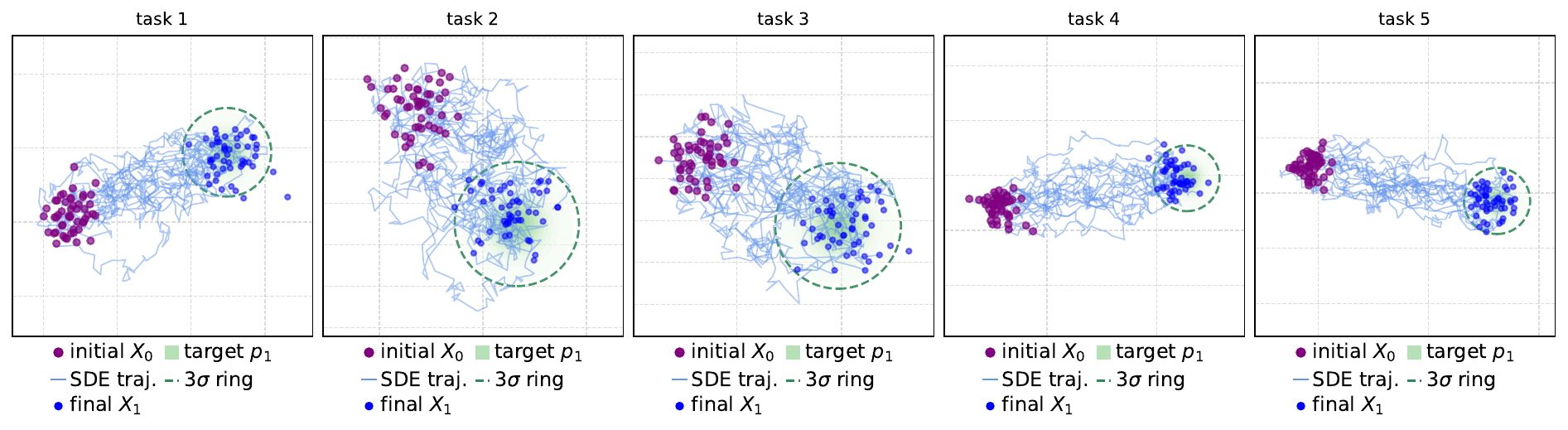}
    \caption{
    One non-parametric Schr\"odinger bridge solution operator ($d=2$, $\g=5.0$),
    solved zero-shot from particle-cloud prompts; panels as in
    Figure~\ref{fig:sb_sweep}: prescribed terminal law $P_1$ (green shading,
    dashed $3\sigma$), trajectories (blue), initial populations (purple).
    The SDE trajectories are simulated according to section~\ref{sec:sde_recovery}.}
    \label{fig:isb_sweep}
\end{figure}

\begin{table}[htpb]
    \centering
    \caption{Non-parametric (particle-cloud) Schr\"odinger bridge vs.\ the exact
    Gaussian reference, $N_{\rm pt}=100$ prompt samples. 
    }
    \label{tab:nonparam_sb_error}
    \begin{tabular}{cccccc}
        \toprule
         &  & \multicolumn{2}{c}{\textbf{Inference Time (s)}}
            & \multicolumn{2}{c}{\textbf{Relative error}} \\
        \cmidrule(lr){3-4}\cmidrule(lr){5-6}
        \textbf{dim} & $\boldsymbol{\g}$ & transport & $+$\,score
            & trajectory & score \\
        \midrule
        $1$  & $1.0$ & 0.015 & 1.320 & 6.70e-2 $\pm$ 4.35e-2 & 1.53e-1 $\pm$ 7.97e-2 \\
        $1$  & $5.0$ & 0.014 & 0.656 & 7.52e-2 $\pm$ 4.84e-2 & 1.70e-1 $\pm$ 1.01e-1 \\
        \midrule
        $2$  & $1.0$ & 0.014 & 0.865 & 4.15e-2 $\pm$ 2.34e-2 & 1.31e-1 $\pm$ 5.48e-2 \\
        $2$  & $5.0$ & 0.014 & 1.308 & 3.93e-2 $\pm$ 2.13e-2 & 9.66e-2 $\pm$ 3.90e-2 \\
        \midrule
        $5$  & $1.0$ & 0.014 & 0.992 & 4.60e-2 $\pm$ 1.77e-2 & 1.24e-1 $\pm$ 3.76e-2 \\
        $8$  & $1.0$ & 0.014 & 0.626 & 5.21e-2 $\pm$ 1.40e-2 & 1.26e-1 $\pm$ 2.88e-2 \\
        $11$ & $1.0$ & 0.014 & 1.359 & 6.10e-2 $\pm$ 1.54e-2 & 1.30e-1 $\pm$ 2.77e-2 \\
        $14$ & $1.0$ & 0.014 & 1.296 & 6.20e-2 $\pm$ 1.28e-2 & 1.25e-1 $\pm$ 2.15e-2 \\
        \bottomrule
    \end{tabular}
\end{table}

\subsection{Systemic Risk Problem}
\label{sec:exp_sr1}

We next consider the systemic-risk model, in which each bank's log-monetary reserve $X_t$ follows~\eqref{eq:Sysrisk_state} and minimizes the social cost~\eqref{eq:Sysrisk_cost}. The task is parameterized by $\theta=(a,q,\varepsilon,c,\sigma)$, where $a$ is the mean-reversion rate, $q$ the control--deviation coupling, $\varepsilon$ and $c$ the running and terminal penalties, and $\sigma$ the idiosyncratic volatility ($\gamma=\sigma^2/2$). During training, $\theta$ is sampled log-uniformly from $[\theta_{\rm base}/3,\,3\,\theta_{\rm base}]$ with $\theta_{\rm base}=\{a=0.1,\,q=0.5,\,\varepsilon=1,\,c=1,\,\sigma=0.5\}$ and $q$ clipped to $0.95\sqrt{\varepsilon}$ to ensure well-posedness ($q<\varepsilon^2$). The closed-form optimal solution is summarized in~\ref{sec:appendix_exact_sr}.

The operator is conditioned directly on the parameter vector $\theta_{\rm base}$ through the conditioning encoder of Section~\ref{sec:arch}. Figure~\ref{fig:sr_fields} compares the learned solution with the analytical reference. The recovered score accurately matches the exact Gaussian score, while the predicted particle distribution follows the contracting Gaussian density throughout the evolution. Figures~\ref{fig:sr_sweep_lo} and~\ref{fig:sr_sweep_hi} further illustrate zero-shot generalization: varying the volatility parameter $\sigma$ changes both the density evolution and the score as expected, without retraining. 

Table~\ref{tab:sr_error} reports quantitative results obtained by varying one component of $\theta$ across its training range while fixing the remaining parameters at $\theta_{\rm base}$. The operator is robust against variations in $a$, $q$, $\varepsilon$, and $\sigma$, with trajectory errors remaining around $10^{-3}$ and score errors around $10^{-2}$. The terminal penalty $c$ has the strongest influence, approximately doubling both errors and producing noticeably larger variability. This suggests that $c$ is the most challenging parameter to amortize and would benefit from denser sampling during training.

\begin{table}[!htpb]
    \centering
    \caption{Systemic risk vs.\ the closed-form Riccati/Gaussian reference:
    per-task inference time and the relative errors~\eqref{equ:err_traj}
    and~\eqref{equ:err_score}.
    Each row sweeps one coordinate of the cost vector
    $\theta=(a,q,\varepsilon,c,\sigma)$ log-uniformly over its training range,
    with the other four held at $\theta_{\rm base}=\{a=0.1,\,q=0.5,\,\varepsilon=1,\,c=1,\,\sigma=0.5\}$;
    All results are from the same trained solution operator network, so inference time is constant.}
    \label{tab:sr_error}
    \begin{tabular}{ccccc}
        \toprule
         & \multicolumn{2}{c}{\textbf{Inference Time (s)}}
            & \multicolumn{2}{c}{\textbf{Relative error}} \\
        \cmidrule(lr){2-3}\cmidrule(lr){4-5}
        \textbf{Parameters} & transport & $+$\,score
            & trajectory & score \\
        \midrule
        $a\in[0.033,\,0.300]$ & 0.012 & 0.901 & 6.36e-3 $\pm$ 2.15e-3 & 5.49e-2 $\pm$ 1.76e-2 \\
        $q\in[0.167,\,0.950]$ & 0.012 & 0.901 & 6.44e-3 $\pm$ 2.33e-3 & 5.56e-2 $\pm$ 1.76e-2 \\
        $\varepsilon\in[0.333,\,3.000]$ & 0.012 & 0.901 & 6.60e-3 $\pm$ 2.57e-3 & 5.56e-2 $\pm$ 1.85e-2 \\
        $c\in[0.333,\,3.000]$ & 0.012 & 0.901 & 1.29e-2 $\pm$ 8.19e-3 & 1.17e-1 $\pm$ 5.18e-2 \\
        $\sigma\in[0.167,\,1.500]$ & 0.012 & 0.901 & 6.76e-3 $\pm$ 3.23e-3 & 5.23e-2 $\pm$ 1.88e-2 \\
        \bottomrule
    \end{tabular}
\end{table}

\begin{figure}[!htpb]
    \centering
    \setlength{\abovecaptionskip}{3pt}
    \includegraphics[width=\linewidth]{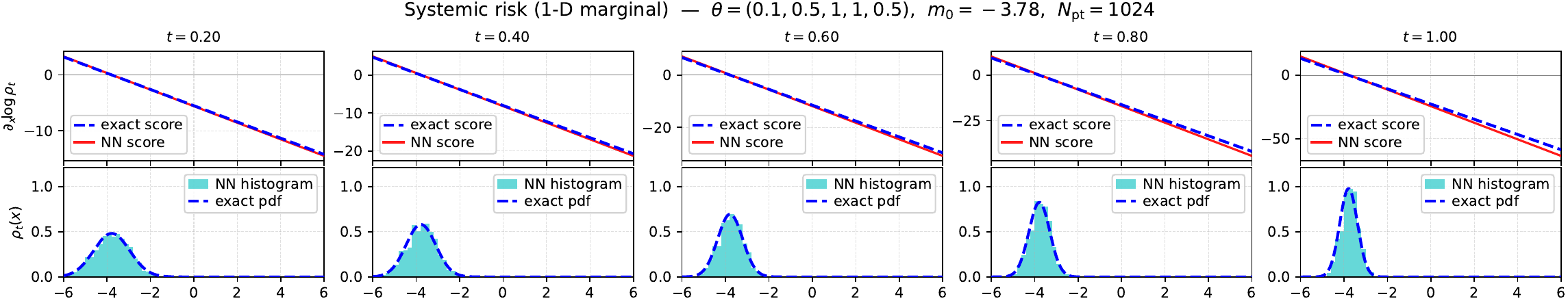}\\[3pt]
    \includegraphics[width=\linewidth]{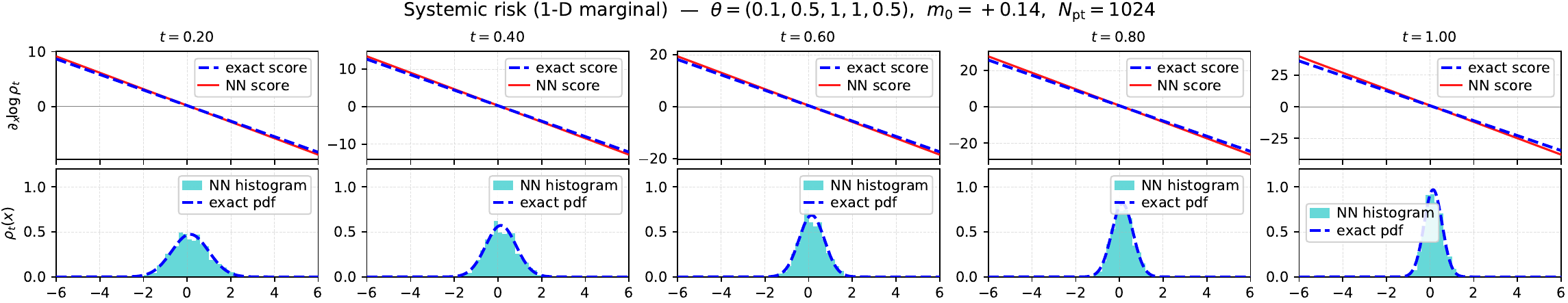}\\[3pt]
    \includegraphics[width=\linewidth]{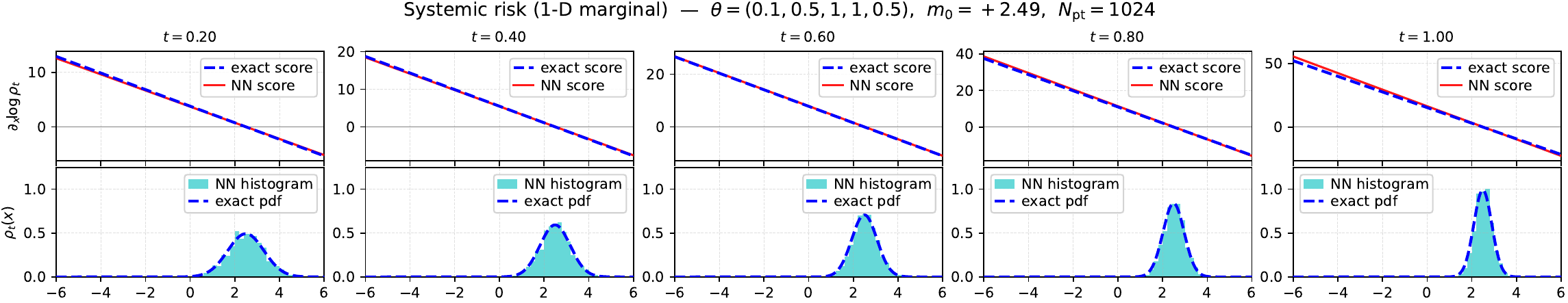}
    \caption{Systemic risk at the base cost $\theta_{\rm base}=\{a=0.1,\,q=0.5,\,\varepsilon=1,\,c=1,\,\sigma=0.5\}$, three
    tasks with initial mean $m_0\in\{-3.8,\, 0.1,\, 2.5\}$
    (top/middle/bottom strips). }
    \label{fig:sr_fields}
\end{figure}

\begin{figure}[!htpb]
    \centering
    \setlength{\abovecaptionskip}{3pt}
    \includegraphics[width=\linewidth]{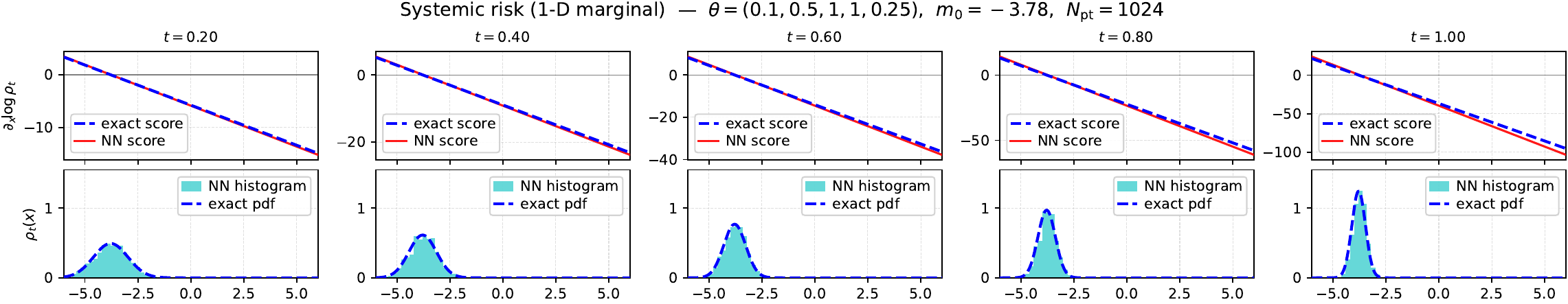}\\[3pt]
    \includegraphics[width=\linewidth]{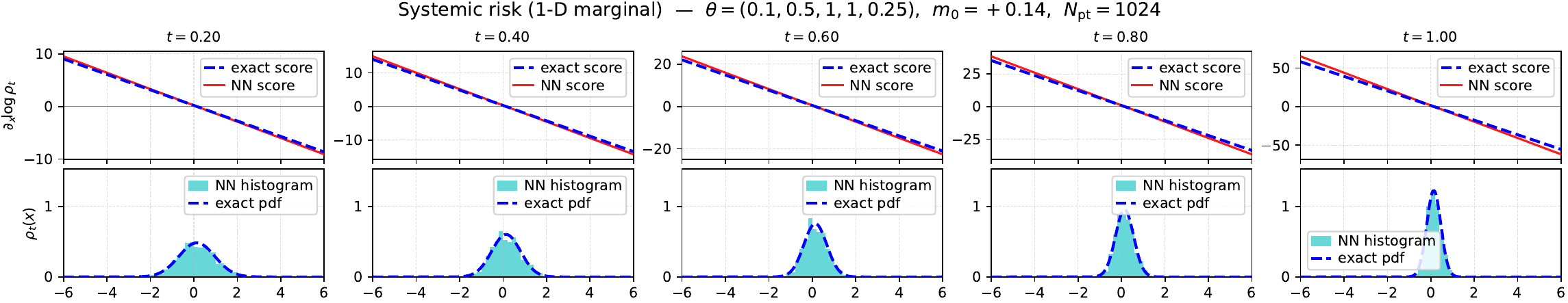}\\[3pt]
    \includegraphics[width=\linewidth]{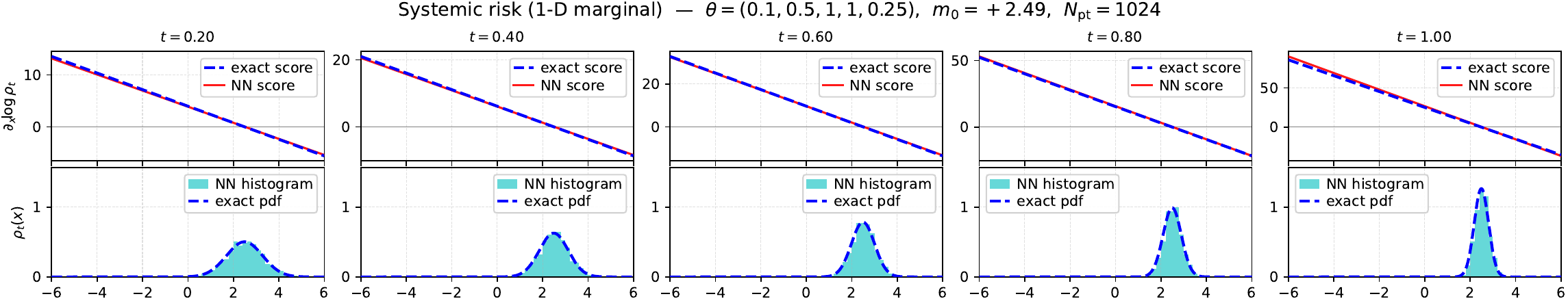}
    \caption{The same operator prompted zero-shot with the low-volatility cost
    $\sigma=0.25$, on the tasks and panels of Figure~\ref{fig:sr_fields}: in
    every task the density concentrates strongly and the score steepens.
    Contrast the high-volatility prompt of Figure~\ref{fig:sr_sweep_hi}.}
    \label{fig:sr_sweep_lo}
\end{figure}

\begin{figure}[!htpb]
    \centering
    \setlength{\abovecaptionskip}{3pt}
    \includegraphics[width=\linewidth]{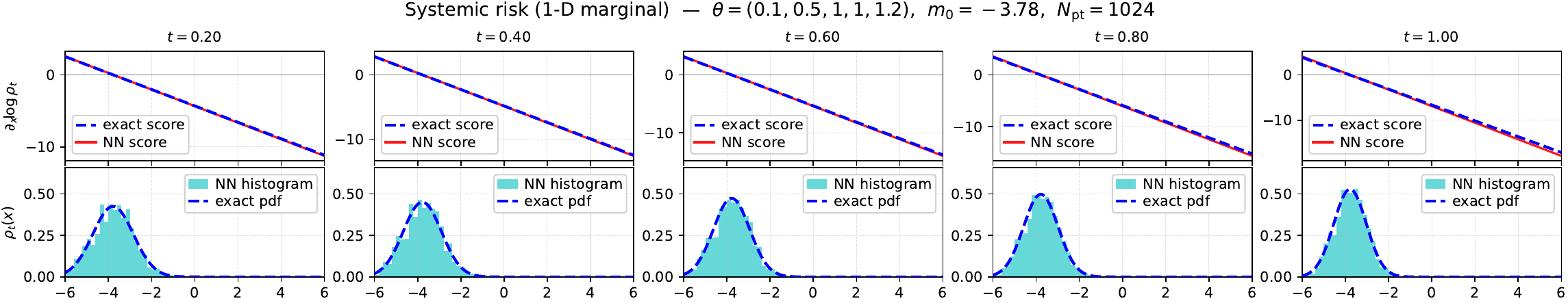}\\[3pt]
    \includegraphics[width=\linewidth]{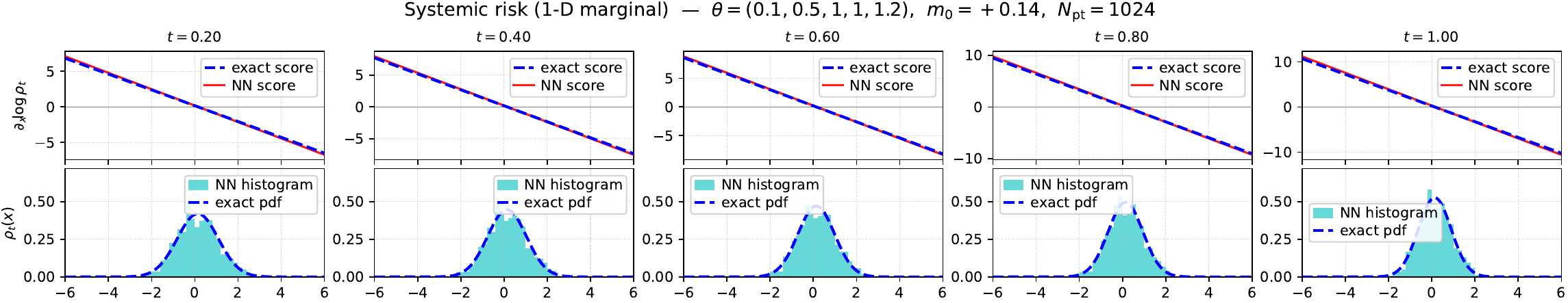}\\[3pt]
    \includegraphics[width=\linewidth]{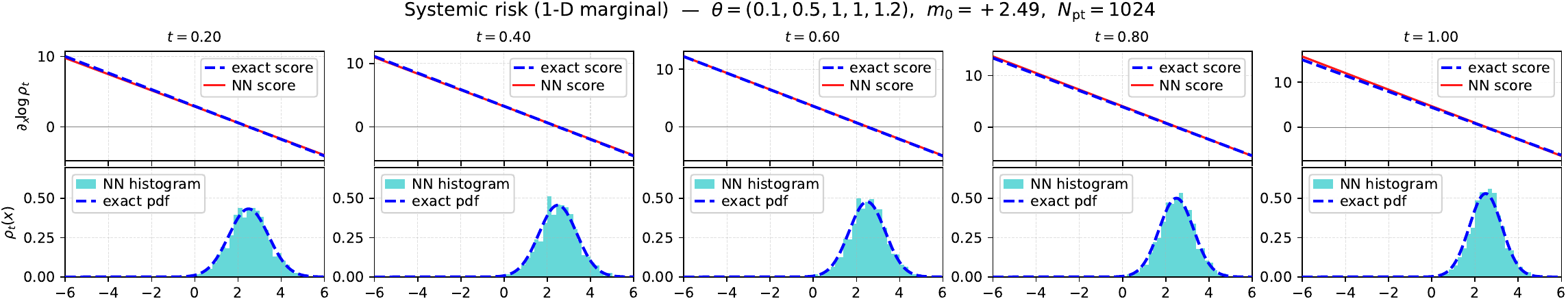}
    \caption{The same operator prompted with the high-volatility cost
    $\sigma=1.2$, on the tasks and panels of Figure~\ref{fig:sr_fields}: with
    no retraining, the density stays markedly wider and the score gentler,
    exactly as the variance ODE prescribes.}
    \label{fig:sr_sweep_hi}
\end{figure}

\subsection{Obstacle-Avoiding Path Planning: Task Generalization}
\label{sec:exp_pp_task}

We finally consider the obstacle-avoiding path-planning problem of Section~\ref{sec:mfc}, where each task is specified by the initial and terminal distributions together with an obstacle field, $\sT=(P_0,P_T,Q)$. Each obstacle is Gaussian, so the field is the thresholded mixture
\begin{align}
    Q(x)=\sum_{j=1}^{n_{\rm obs}}
    \max\bigl\{\,\sN(x;c_j,\Sigma_j)-\tau,\;0\,\bigr\},
    \label{equ:obstacle_field}
\end{align}
where $\sN(\cdot\,;c_j,\Sigma_j)$ denotes the Gaussian density with center $c_j$ and covariance $\Sigma_j$, and $\tau=10^{-3}$ truncates the tails so that each obstacle penalizes only its own neighborhood rather than the whole domain. The interaction term of Section~\ref{sec:mfc} is the induced occupancy cost $I(p(\cdot,t))=\int_\Omega Q(x)\,p(x,t)\dd x$, accumulated over the time grid. We use $n_{\rm obs}=3$ obstacles, with centers and covariances resampled per task. With the cost weights fixed, the operator is evaluated on unseen obstacle configurations, requiring a single network to adapt its transport to different environments through the task prompt alone.

Figure~\ref{fig:pp_task} illustrates zero-shot generalization on tasks. The learned transport successfully routes the particle population around the prompted obstacles while reaching the prescribed terminal distribution, without retraining. The predicted trajectories are further refined by a lightweight quadratic-programming projection based on control barrier functions, which minimally adjusts the transport to satisfy the prescribed $3\sigma$ keep-out constraints. 
\begin{figure}[!ht]
    \centering
    \includegraphics[width=0.185\linewidth]{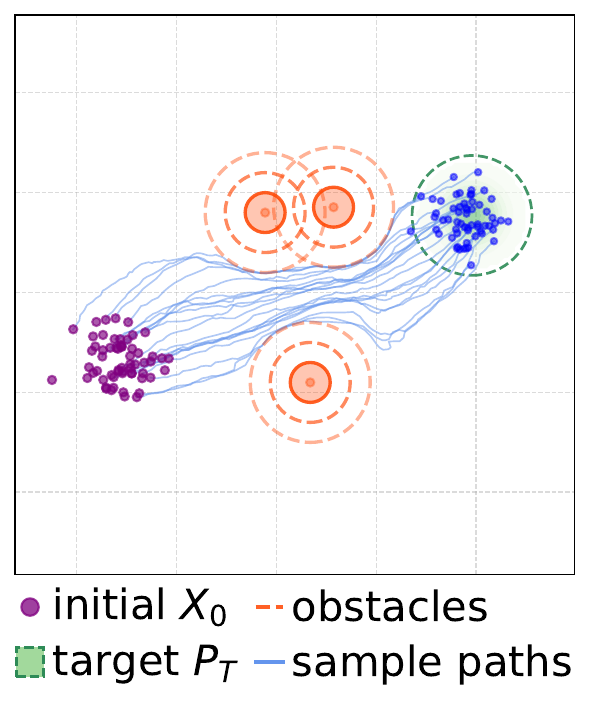}
    \includegraphics[width=0.185\linewidth]{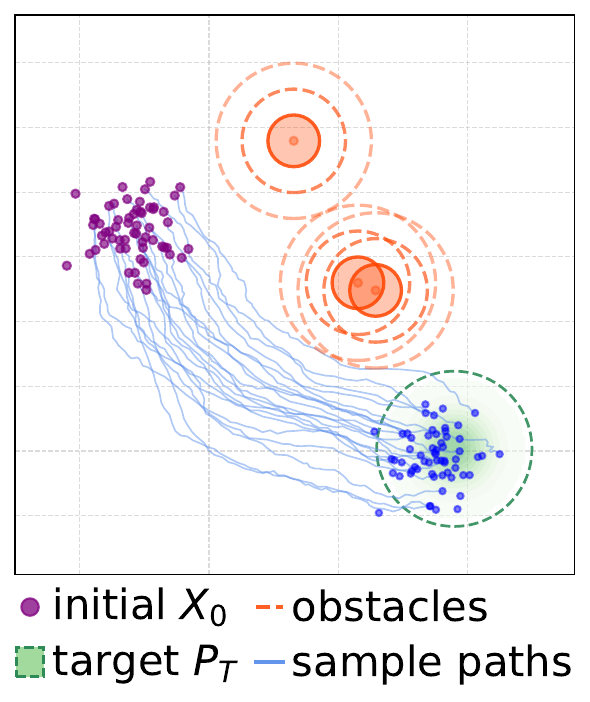}
    \includegraphics[width=0.185\linewidth]{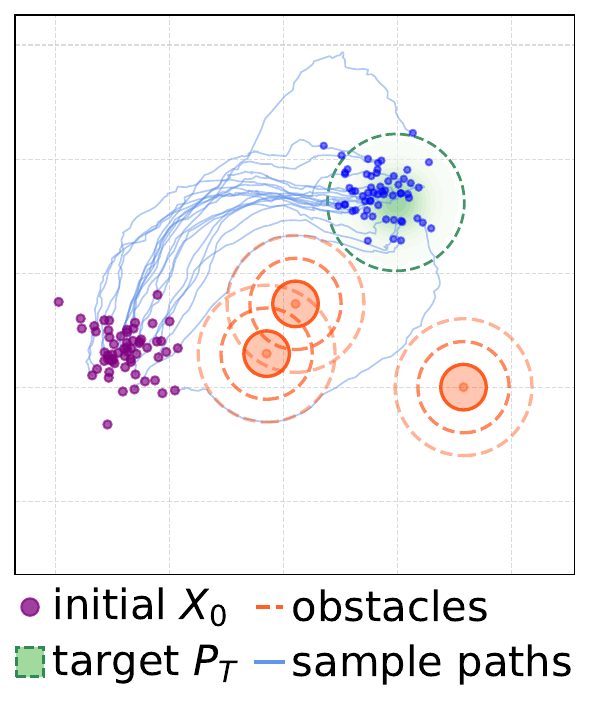}
    \includegraphics[width=0.185\linewidth]{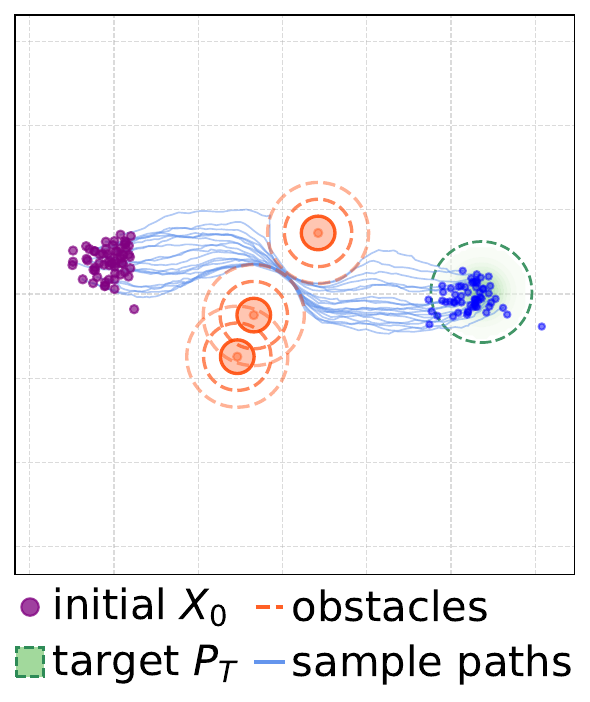}
    \includegraphics[width=0.185\linewidth]{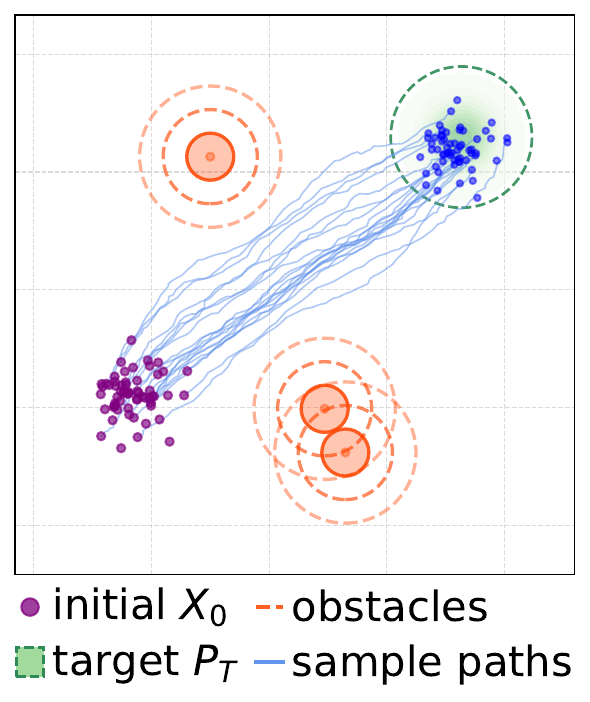}
    \\
    \includegraphics[width=0.185\linewidth]{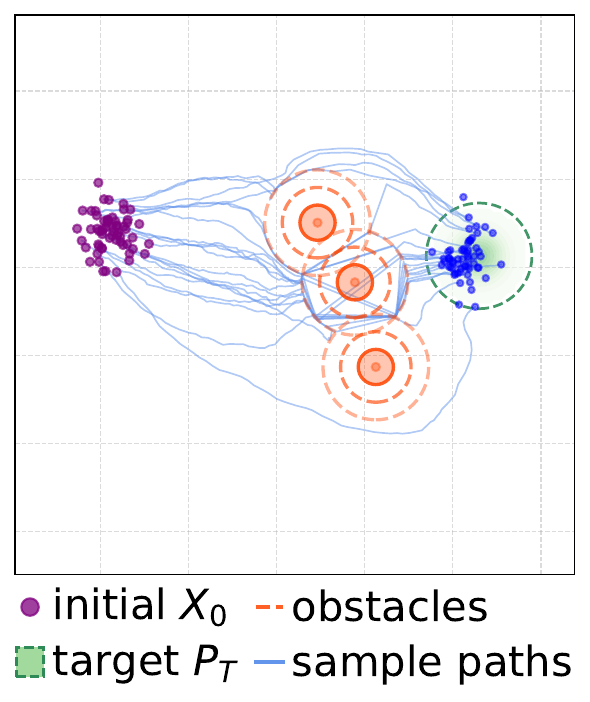}
    \includegraphics[width=0.185\linewidth]{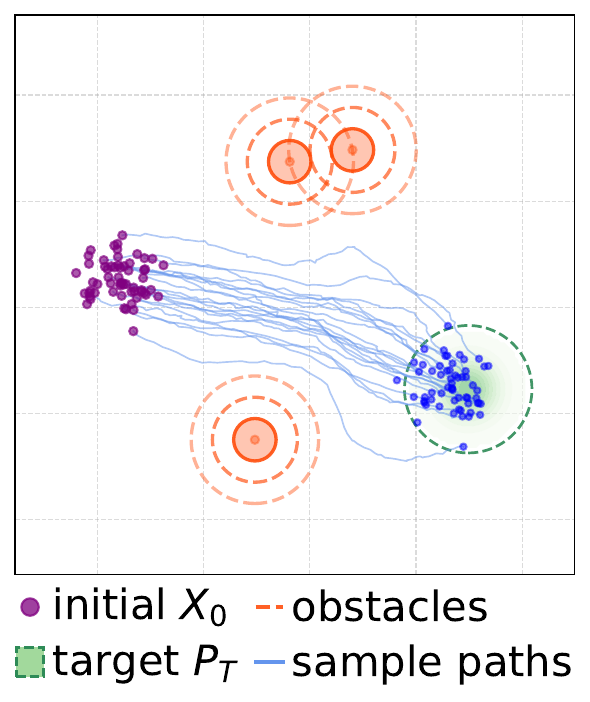}
    \includegraphics[width=0.185\linewidth]{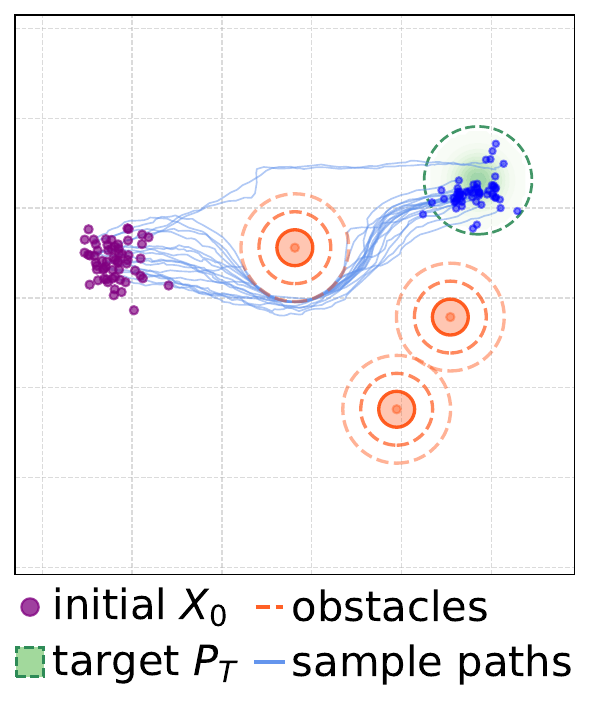}
    \includegraphics[width=0.185\linewidth]{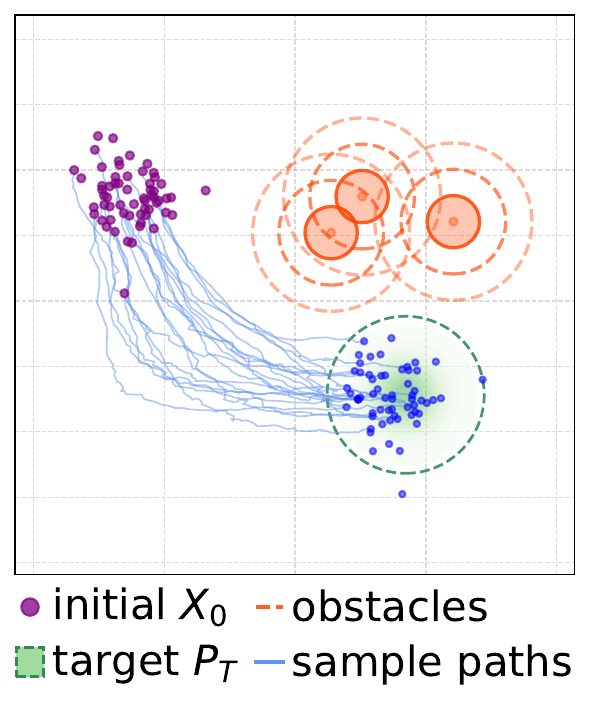}
    \includegraphics[width=0.185\linewidth]{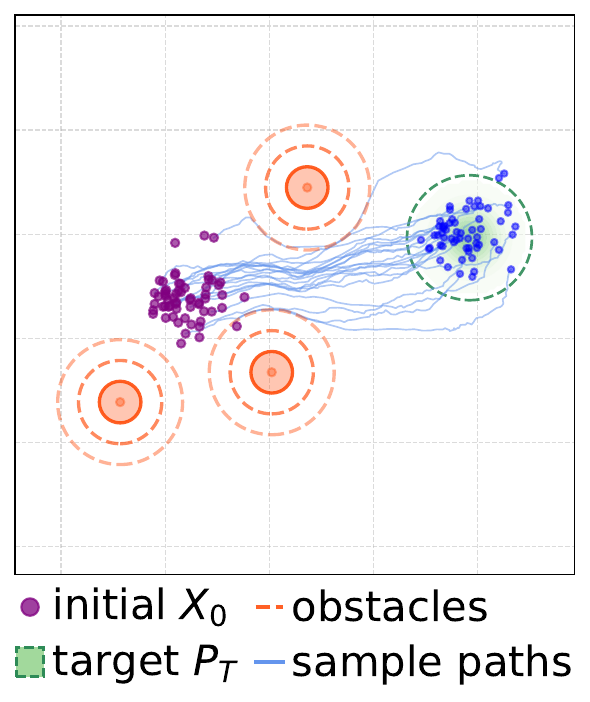}
    \\
    \includegraphics[width=0.185\linewidth]{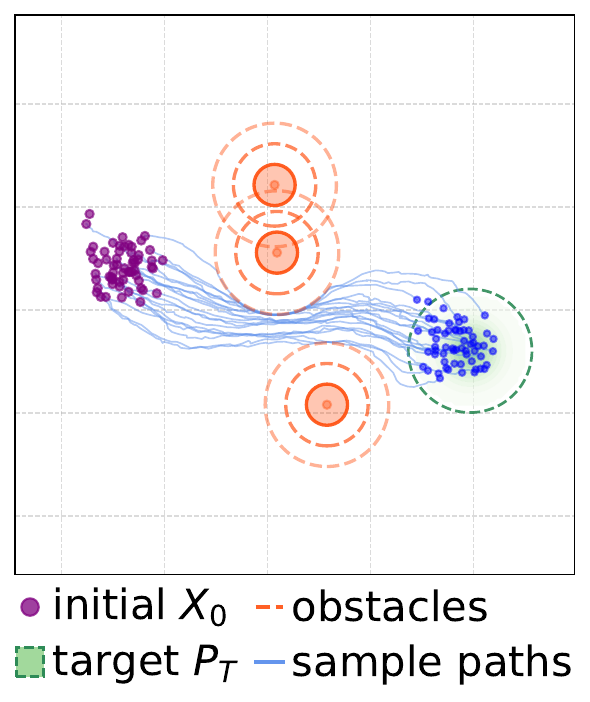}
    \includegraphics[width=0.185\linewidth]{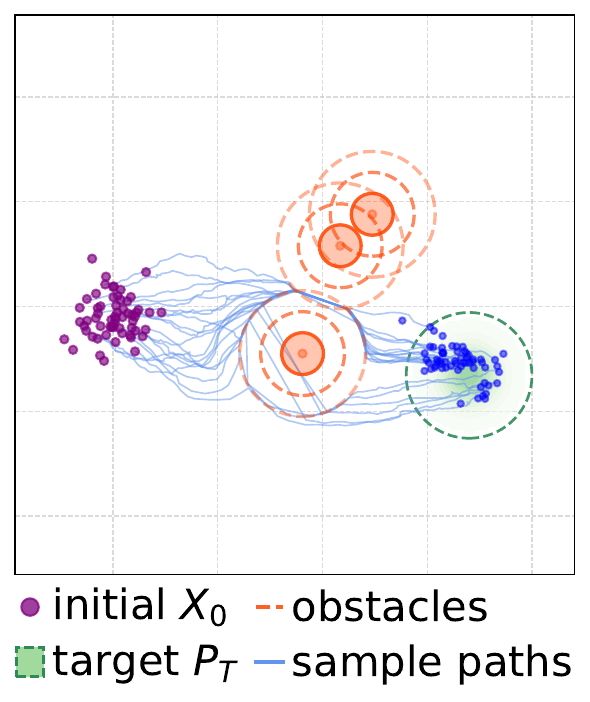}
    \includegraphics[width=0.185\linewidth]{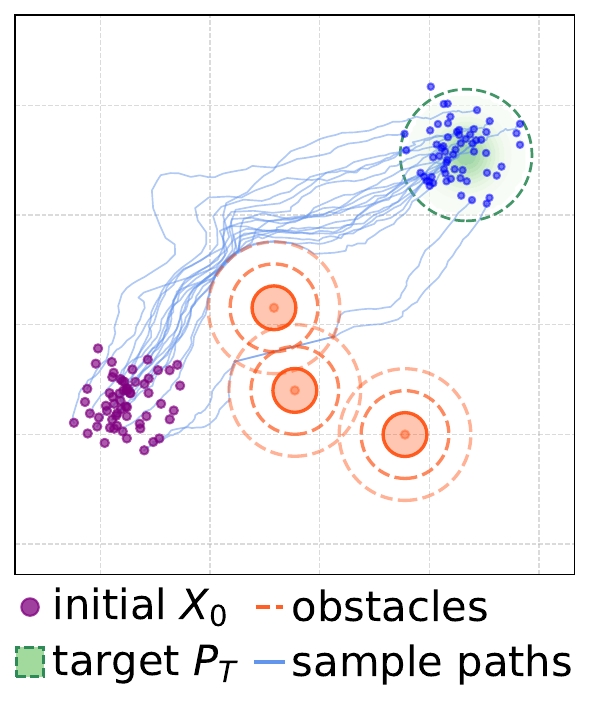}
    \includegraphics[width=0.185\linewidth]{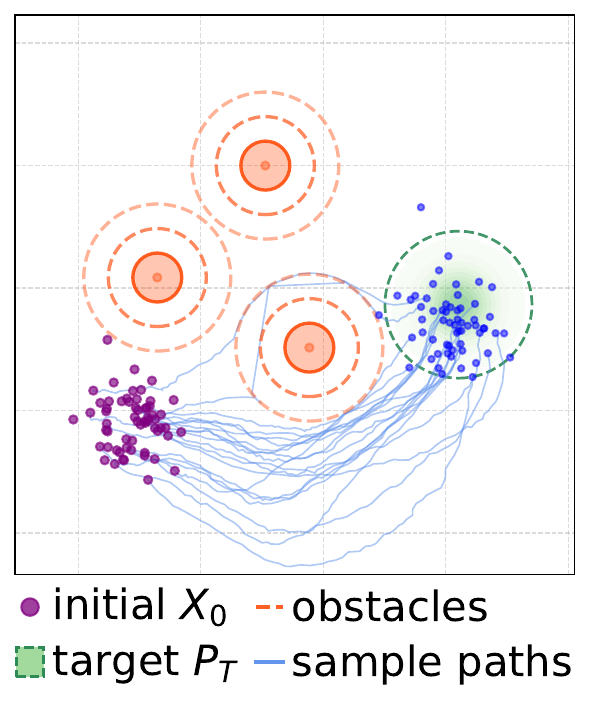}
    \includegraphics[width=0.185\linewidth]{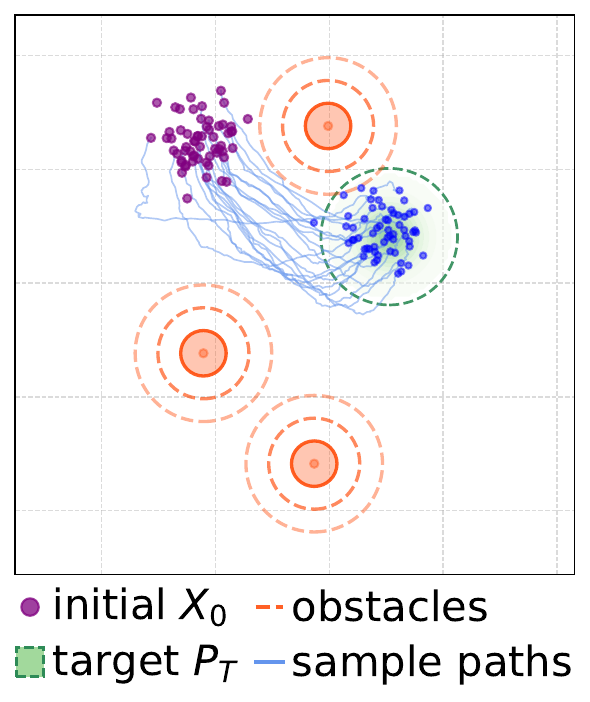}
    \caption{One path-planning solution operator, zero-shot on 
    configurations: initial distribution $P_0$ (purple), target distribution $P_T$ (green, dashed $3\sigma$), obstacles (orange rings). 
    Blue curves are Euler--Maruyama sample
    paths of the controlled diffusion~\eqref{equ:sde} at the trained $\g=0.1$; the quadratic programming projection enforces the $3\sigma$ keep-out.}
    \label{fig:pp_task}
\end{figure}

\subsection{Inference Cost and Dimension Scaling}
\label{sec:exp_complexity}

In Section~\ref{sec:arch}, we argue that transport, the Jacobian log-determinant and the analytic score each cost $\mathcal O(d)$ per particle and per time node, against the $\mathcal O(d^3)$ of a generic flow map. The experiments above span $d=1$--$14$, a range over which those three rates are not separable from a constant, so we measure the cost directly over five orders of magnitude in $d$.

This inference cost test is purely architectural. We evaluate an untrained solution-operator network with the same architecture as the multi-well operator in Section~\ref{sec:exp_soc_mw}, varying only the state dimension \(d\). This test can empirically assess the predicted  $\mathcal O(d)$ scaling of the inference cost with respect to dimension. Table~\ref{tab:dim_timing} and Figure~\ref{fig:dim_timing} report the result: both costs are constant while the $d$-independent work dominates and strictly linear thereafter, never approaching the quadratic or cubic guides. We end the experiment at $d=262144$ where the $24$~GB GPU VRAM is exhausted. 

This experiment does not imply that a network of fixed size has sufficient approximation capacity to solve MFC problems in arbitrarily high dimensions. Characterizing the network capacity required as the dimension and problem complexity increase is a separate question related to approximation and generalization theory, and is beyond the scope of the present computational scaling experiment. We leave a systematic investigation of this question to future work. 

\begin{table}[htpb]
    \centering
    \caption{Per-task inference cost of the multi-well operator versus state
    dimension, on untrained (randomly initialized) models. ``transport'' is
    one batched rollout over all $n_{\rm time}$ nodes and ``score'' is one
    analytic-score evaluation at each node. Everything but $d$ is fixed:
    $B=1$, $N_{\rm pt}=100$, $n_{\rm time}=32$.}
    \label{tab:dim_timing}
    \begin{tabular}{ccccc}
        \toprule
         & \multicolumn{2}{c}{\textbf{Inference Time (s)}}
            & \textbf{Peak memory} & \textbf{Parameters} \\
        \cmidrule(lr){2-3}
        \textbf{dim} & transport & score & (GB) & ($10^6$) \\
        \midrule
        $2$ & 0.0095 & 0.6302 & 0.06 & 4.55 \\
        $4$ & 0.0095 & 0.6177 & 0.06 & 4.56 \\
        $8$ & 0.0095 & 0.6192 & 0.06 & 4.56 \\
        $16$ & 0.0095 & 0.6184 & 0.06 & 4.58 \\
        $32$ & 0.0095 & 0.6197 & 0.06 & 4.61 \\
        $64$ & 0.0104 & 0.6859 & 0.06 & 4.66 \\
        $128$ & 0.0106 & 0.6888 & 0.06 & 4.77 \\
        $256$ & 0.0106 & 0.6836 & 0.06 & 4.98 \\
        $512$ & 0.0106 & 0.6898 & 0.08 & 5.41 \\
        $1024$ & 0.0107 & 0.6883 & 0.12 & 6.27 \\
        $2048$ & 0.0097 & 0.6945 & 0.20 & 7.99 \\
        $4096$ & 0.0167 & 0.6147 & 0.36 & 11.43 \\
        $8192$ & 0.0329 & 0.6704 & 0.68 & 18.31 \\
        $16384$ & 0.0644 & 0.6728 & 1.32 & 32.08 \\
        $32768$ & 0.1294 & 0.6730 & 2.60 & 59.60 \\
        $65536$ & 0.2745 & 0.8472 & 5.16 & 114.65 \\
        $131072$ & 0.5874 & 2.0654 & 10.29 & 224.75 \\
        $262144$ & 1.2402 & 4.4567 & 20.54 & 444.95 \\
        \bottomrule
    \end{tabular}
\end{table}

\begin{figure}[!htpb]
    \centering
    \includegraphics[width=0.62\linewidth]{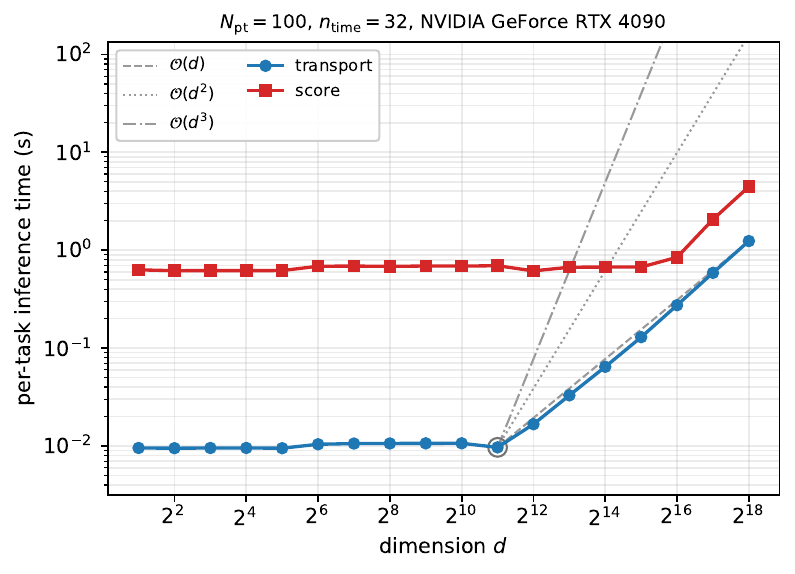}
    \caption{Per-task inference cost versus dimension for the multi-well
    operator, at $B=1$, $N_{\rm pt}=100$ and $n_{\rm time}=32$. The grey guides
    are $\mathcal O(d)$, $\mathcal O(d^2)$ and $\mathcal O(d^3)$, anchored at the
    transport measurement at $d=2048$ (circled), where growth begins, so that
    each guide is compared with the data over the range in which an asymptotic
    rate is claimed.}
    \label{fig:dim_timing}
\end{figure}

\section{Conclusions}\label{sec:conclusion}
We develop a self-supervised operator-learning framework for solving families of stochastic mean-field control problems. By reformulating the Fokker--Planck equation as the probability flow ODE, we represent the density evolution by a deterministic transport map while recovering the score required by the original mean-field control problem. To make this representation computationally tractable, we introduced NFIST, a task-conditioned invertible architecture that combines coupling-based normalizing flows with transformer conditioning. Its exact inverse and analytic log-determinant enable efficient score evaluation, while parametric and particle-cloud prompts allow a single trained network to act as a solution operator over previously unseen tasks. Training is fully self-supervised through the stochastic control objective and requires no reference solutions. Numerical experiments across stochastic optimal control, Schr\"odinger bridges, systemic risk, and obstacle-avoiding path planning demonstrate accurate generalization across distributions, model parameters, state dimensions, and task configurations. These results establish invertible probability-flow learning as an effective framework for learning solution operators for stochastic mean-field control.

\section*{Declaration of AI-assisted technologies in the manuscript preparation process}
During the preparation of this work the authors used Claude (Anthropic) and ChatGPT (OpenAI) in order to assist with the aspects of the numerical implementation and to improve the language and readability of the manuscript. After using this tool, the authors reviewed and edited the content as needed and take full responsibility for the content of the published article.

{\footnotesize
\bibliographystyle{elsarticle-num}
\bibliography{ref}
}

\appendix

\section{Proofs}\label{sec:appendix_proofs}
\subsection{Proof of Lemma~\ref{lem:prob_flow}}
\label{sec:proof_prob_flow}

The density $p$ of the SDE~\eqref{equ:sde} satisfies the Fokker--Planck equation
\begin{align}
    \partial_t p+\nabla\!\cdot(pv)=\g\Delta p,
    \qquad
    p(\cdot,0)=p_0.
    \label{equ:proof_fp}
\end{align}
Using $f=v-\g\nabla\log p$ and the identity $p\nabla\log p=\nabla p$, equation~\eqref{equ:proof_fp} becomes
\begin{align*}
    0
    &=\partial_t p+\nabla\!\cdot
      \bigl(p(f+\g\nabla\log p)\bigr)-\g\Delta p \\
    &=\partial_t p+\nabla\!\cdot(pf).
\end{align*}
Thus, $p$ satisfies the continuity equation associated with the probability-flow ODE
\[
    \partial_t x_t=f(x_t,t),
    \qquad
    x_0\sim P_0.
\]
The density $\tilde p$ of this ODE satisfies the same continuity equation and initial condition. By uniqueness of its solution, $\tilde p(\cdot,t)=p(\cdot,t)$ for all $t\in[0,T]$.
\qed


\subsection{Proof of Lemma~\ref{lem:score_pushforward}}
\label{sec:proof_score_pushforward}

Let $x_t=G(x,t)$. Since $G(\cdot,t)_\#p_0=p(\cdot,t)$, the change-of-variables formula gives
\begin{align}
    \log p_0(x)
    =
    \log p(G(x,t),t)
    +
    \log\left|\det\boldsymbol{\nabla}_xG(x,t)\right|.
    \label{equ:proof_change_variables}
\end{align}
Differentiating with respect to $x$ and applying the chain rule yields
\begin{align*}
    s_0(x)
    &=
    \boldsymbol{\nabla}_xG(x,t)^{\mathsf T}
    \mathbf{s}(G(x,t),t)
    +
    \nabla_x
    \log\left|\det\boldsymbol{\nabla}_xG(x,t)\right|.
\end{align*}
Because $G(\cdot,t)$ is a diffeomorphism, $\boldsymbol{\nabla}_xG(x,t)$ is invertible. Therefore,
\begin{align*}
    \mathbf{s}(G(x,t),t)
    =
    \boldsymbol{\nabla}_xG(x,t)^{-\mathsf T}
    \left[
        s_0(x)
        -
        \nabla_x
        \log\left|\det\boldsymbol{\nabla}_xG(x,t)\right|
    \right].
\end{align*}
\qed

\section{Exact Solutions}\label{sec:appendix_exact_sol}

\subsection{Stochastic Optimal Control with Gaussian Data}
\label{sec:appendix_soc}

We consider the stochastic optimal control problem
\begin{align*}
J(v)
&=
\mathbb E\!\left[
\frac12\int_0^1\|v(X_t,t)\|^2\,\dd t
+
\frac12\|X_1-m_1\|^2
\right]\\
\rd X_t&=v(X_t,t)\,\rd t+\sqrt{2\gamma}\,\rd W_t,
\qquad
X_0\sim\mathcal N(m_0,aI),
\end{align*}
This problem admits the explicit optimal control $v^*(x,t)=\frac{m_1-x}{2-t}$, under which the state remains Gaussian $X_t\sim\mathcal N(\mu(t),\Sigma(t))$, with
\[
\mu(t)=\frac{2-t}{2}m_0+\frac{t}{2}m_1,\quad 
\Sigma(t)=s(t)I,\quad s(t)=\frac{(2-t)^2}{4}a+\gamma\,t(2-t).
\]
The probability-flow map transporting the initial distribution to the marginal at time \(t\) is
\[
G(x,t)
=
\alpha(t)x
+
\mu(t)
-
\alpha(t)m_0,
\qquad
\alpha(t)
=
\sqrt{\frac{s(t)}{a}},
\]
which satisfies
\[
G(\cdot,t)_\#\mathcal N(m_0,aI)
=
\mathcal N(\mu(t),\Sigma(t)).
\]

\subsection{Multi-Well Stochastic Optimal Control}\label{sec:appendix_soc_mw}

The experiment of Section~\ref{sec:exp_soc_mw} replaces the quadratic terminal cost of \ref{sec:appendix_soc} by the soft-min potential \(V(x)=-\log\sum_{i=1}^{n}e^{-\norm{x-\m_i}^2}\), for which no closed form is available. The control problem
\[
\min_v\;
\E\left[\tfrac12\int_0^1\norm{v(X_t,t)}^2\dd t+\tfrac12 V(X_1)\right],
\qquad
\rd X_t=v(X_t,t)\,\rd t+\sqrt{2\g}\,\rd W_t,
\quad X_0\sim P_0,
\]
is nevertheless linearly solvable. Its Hamilton--Jacobi--Bellman equation for the value function $\psi$, $\partial_t\psi-\tfrac12\norm{\nabla\psi}^2+\g\Delta\psi=0$ with $\psi(\cdot,1)=\tfrac12 V$, linearizes under the Hopf--Cole transform $\psi=-2\g\log\eta$ \cite{fleming2006controlled,kappen2005linear} into a backward heat equation for
the desirability $\eta$,
\[
\partial_t\eta+\g\Delta\eta=0,
\qquad
\eta(\cdot,1)=e^{-V/(4\g)},
\]
whose solution is the Gaussian convolution
\begin{align}
\eta(x,t)
=
\qb{4\pi\g(1-t)}^{-d/2}
\int_{\R^d}
\exp\qB{-\frac{\norm{x-y}^2}{4\g(1-t)}-\frac{V(y)}{4\g}}\dd y
=
\E_{Z\sim\sN\qb{0,\,2\g(1-t)I}}\left[e^{-V(x+Z)/(4\g)}\right].
\label{equ:hopfcole_eta}
\end{align}
The optimal control is $v^*=-\nabla\psi=2\g\nabla\log\eta$, and the controlled marginal factorizes as $p_t=\eta(\cdot,t)\,\hat\eta(\cdot,t)$, where $\hat\eta$ solves the forward heat equation~\cite{daipra1991stochastic,chen2021liaisons}
\[
\partial_t\hat\eta=\g\Delta\hat\eta,
\qquad
\hat\eta(\cdot,0)=p_0/\eta(\cdot,0),
\qquad\text{that is,}\quad
\hat\eta(\cdot,t)=\E_{Z\sim\sN\qb{0,\,2\g t I}}\left[\hat\eta(\cdot+Z,0)\right].
\]
Both fields are thus Gaussian smoothings of fixed data, and $\nabla\log p_t=\nabla\log\eta+\nabla\log\hat\eta$.

Because $V/(4\g)$ carries a non-integer power of the mixture, the two convolutions are evaluated numerically on a $161^2$ tensor grid, in the log domain by a separable log-sum-exp that stays stable at small $\g$, with second-order central differences for the gradients. Reference trajectories follow the probability flow ODE of Lemma~\ref{lem:prob_flow},
\begin{align}
\dot x^*=v^*(x^*,t)-\g\nabla\log p_t(x^*),
\label{equ:hopfcole_pfode}
\end{align}
integrated with six Euler substeps per recording interval. Two reductions validate the solver.

\subsection{Explicit Solution for Schr\"odinger Bridges with Gaussian Endpoints}
\label{sec:appendix_sb}

The Schr\"odinger bridge seeks the most likely diffusion connecting two prescribed endpoint distributions. Throughout this paper we consider the Brownian reference process
\[
\rd X_t=\sqrt{2\gamma}\,\rd W_t,
\]
where $\gamma>0$ is the diffusion coefficient. By Girsanov's theorem (see, e.g.,~\cite{chen2016relation}), the problem is equivalent to
\begin{align}
\min_{(p,f)}
\int_0^1\!\!\int_{\mathbb R^d}
\left[
\frac12\|f(x,t)\|^2
+\frac{\gamma^2}{2}
\|\nabla\log p(x,t)\|^2
\right]
p(x,t)\,\dd x\,\dd t,
\end{align}
subject to the continuity equation
\[
\partial_tp+\nabla\!\cdot(p f)=0,
\qquad
p(\cdot,0)=p_0,\;
p(\cdot,1)=p_1.
\]

For isotropic Gaussian marginals $P_0=\mathcal N(m_0,aI),    P_1=\mathcal N(m_1,bI)$, the optimal solution admits a closed form $P_t=\mathcal N(\mu(t), \Sigma(t))$~\cite{chen2016relation}. Define $c_0= 1+\frac{\gamma}{a}-\frac1a\sqrt{\gamma^2+ab}$. Then
\[
\mu(t)=(1-t)m_0+tm_1, \qquad 
\Sigma(t)
=
\Bigl(
(1-c_0t)^2a
+2\gamma(1-c_0t)t
\Bigr)I.
\]
The corresponding probability-flow velocity field is
\[
f(x,t)
=
\frac{m_1-(1-c_0)m_0}{1-c_0t}
-\frac{c_0}{1-c_0t}x
+\frac12\Sigma(t)^{-1}
\bigl(x-\mu(t)\bigr),
\]
which is used to generate the reference solutions in our numerical experiments.

\subsection{Systemic Risk}\label{sec:appendix_exact_sr}

We consider the systemic-risk example with dynamics 
$\dd X_t=[a(\bar m_t-X_t)+\alpha_t]\dd t+\sigma\dd W_t$
with $\bar m_t=\E[X_t]$, running cost 
$\tfrac12\alpha^2-q\,\alpha(\bar m-x)+\tfrac{\varepsilon}{2}(\bar m-x)^2$, 
and terminal cost $\tfrac{c}{2}(x-\bar m)^2$, in the diffusion convention $\g=\sigma^2/2$. The optimal control is the affine feedback
\[
\alpha^*(t,x) = (q+\eta_t)\,(\bar m_t - x),
\]
where $\eta_t$ solves the scalar Riccati equation
\[
\dot\eta_t = \eta_t^2 + 2(a+q)\,\eta_t - (\varepsilon - q^2),
\qquad \eta_T = c,
\]
available in closed form: with 
$\delta^{\pm}=-(a+q)\pm\sqrt{(a+q)^2+\varepsilon-q^2}$ 
and 
$E_t=e^{(\delta^+-\delta^-)(T-t)}$,
\[
\eta_t
= \frac{-(\varepsilon-q^2)\,(E_t-1) - c\,(\delta^+E_t-\delta^-)}
{(\delta^-E_t-\delta^+) - c\,(E_t-1)}.
\]
Averaging the closed-loop dynamics gives
$\partial_t\bar m_t=0$, so $\bar m_t=\bar m_0=\E_{\mu_0}[X]$, and the marginal remains Gaussian, $\sN(\bar m_0,\Sigma_t)$ per coordinate, with variance ODE
\[
\dot\Sigma_t = -2\kappa_t\,\Sigma_t + \sigma^2,
\qquad \kappa_t = a+q+\eta_t,
\qquad \Sigma_0 = \mathrm{Var}(\mu_0),
\]
whose integrating-factor solution $\Sigma_t = e^{-2K_t}\bigl(\Sigma_0+\sigma^2\int_0^t e^{2K_s}\dd s\bigr)$, $K_t=\int_0^t\kappa_s\dd s$, the implementation evaluates by high-resolution quadrature with the closed-form $\eta_t$; the reference score is $-(x-\bar m_0)/\Sigma_t$. These expressions provide the closed-form reference used in Section~\ref{sec:exp_sr1}; the one-dimensional integrals defining \(K_t\) and \(\Sigma_t\) are evaluated numerically.

\end{document}